\documentclass[11pt]{article}
\usepackage{latexsym,amssymb,amsmath,a4wide,theorem}
\usepackage{graphicx}
\usepackage{subfig}
\usepackage{caption}
\usepackage{color}
\usepackage{extarrows}
\usepackage{esint}

\numberwithin{equation}{section}
\numberwithin{figure}{section}
\newtheorem{theorem}{Theorem}[section]
\newenvironment{proof}{{ \it Proof:\quad}}{\hfill $\blacksquare$\par}
\newtheorem{lemma}{Lemma}[section]
\newtheorem{remark}{Remark}[section]
\newtheorem{proposition}{Proposition}[section]

\title{Nonlinear Stability for the Superposition of Viscous Contact Wave and Rarefaction Waves to Non-isentropic Compressible Navier-Stokes System with General Initial Perturbations}
\date{  }
\author{Yi Peng$^1$, Xiaoding Shi$^{1*}$, Yuhang Wu$^2$\\
\scriptsize$^{1}$ {College of Mathematics and Physics, Beijing University of Chemical Technology, Beijing, 100029, China}\\
\scriptsize$^{2}$ {Department of Mathematics, City University of Hong Kong, Hong Kong, China}
}

\begin{document}
\maketitle

\begin{abstract}
In this paper, the large time behavior of the solutions for the Cauchy problem to the one-dimensional compressible Navier-Stokes system with the motion of a viscous heat-conducting perfect polytropic gas is investigated.
Our result shows that  the combination of a viscous contact wave with rarefaction waves  is asymptotically stable, when the large initial disturbance of the density,  velocity and  temperature belong to $H^{1}(\mathbb{R})$, $L^{2}(\mathbb{R})\cap L^{4}(\mathbb{R})$ and $L^{2}(\mathbb{R})$, provided the strength of the combination waves is suitably small. In addition, the initial disturbance on the derivation of velocity and  temperature belong to $L^{2}(\mathbb{R})$ can be arbitrarily large.
\end{abstract}


\


\

\renewcommand{\thefootnote}{\fnsymbol{footnote}}
\footnotetext[1]{{Corresponding author. The corresponding author is supported by National Natural Science Foundation of China, ( No. 12171024).}\\
{Email address: apengyi@163.com (Y. Peng),  shixd@mail.buct.edu.cn (X. Shi),
yuhangwu2-c@my.cityu.edu.hk (Y. Wu)}}

\section{Introduction}
\indent\qquad
 We consider the one-dimensional compressible Navier-Stokes system in Lagrangian coordinates:
\begin{gather}
v_{t}-u_{x}=0,\label{01}\\
u_{t}+p_{x}=\mu\left(\frac{u_{x}}{v}\right)_{x},\label{02}\\
\big( e+\frac{u^{2}}{2} \big)_{t}+(p u)_{x}=\big( \kappa\frac{\theta_{x}}{v}
+\mu\frac{u u_{x}}{v} \big)_{x},\label{03}
\end{gather}
where $t>0$ is time, $x\in\mathbb{R}=(-\infty,+\infty)$ denotes the Lagrange mass coordinate, and the unknown functions $v(x,t)>0$, $u(x,t)$ and 
 $p(x,t)$ are the specific volume of the gas, fluid velocity and pressure respectively. Here we study the ideal polytropic fluids so that $p$ and $e$ satisfy
\begin{gather}
    p=\frac{R\theta}{v}=Av^{-\gamma}e^{ \frac{\gamma-1}{R}s}, \quad
    e=c_{\nu}\theta+\text{const},
\end{gather}
where $s$ is the entropy, $\gamma>1$ is the adiabatic exponent, $c_{\nu}=R/(\gamma-1)$ is the specific heat, $A$ and $R$ are both positive constants. The Cauchy problem to system \eqref{01}-\eqref{03} supplemented with the initial and far field conditions
\begin{equation}
\left. \begin{array}{l}
\displaystyle (v,u,\theta)(x,0)=(v_{0},u_{0},\theta_{0})(x)\xrightarrow{x\rightarrow\pm\infty}(v_{\pm},u_{\pm},\theta_{\pm}),\quad \ \quad x\in\mathbb{R},\label{23}
       \end{array} \right.\\
\end{equation}
where $v_{\pm}>0$, $u_{\pm}$ and $\theta_{\pm}>0$ are given constants. 

When the far field states satisfy $v_{+}=v_{-},u_{+}=u_{-},\theta_{+}=\theta_{-}$, there is huge literature on the studies of the global existence of solutions to the compressible Navier-Stokes system. Indeed, Kazhikhov and Shelukhin \cite{15a} first obtain the global existence of solutions in bounded domains with large initial data. From then on, significant progress has been made on the mathematical aspect of the initial boundary value problems, for details, see  Duan-Guo-Zhu \cite{5a}, Goodman \cite{8},  Jiang \cite{9}-\cite{10}, Li-Liang \cite{16}, Pan-Zhang \cite{23}, Huang-Shi \cite{6}, Huang-Shi-Sun \cite{7} and the references therein. It needs to be pointed out in particular,  Jiang \cite{9}--\cite{10} first proved that the specific volume is uniformly bounded from below and above in unbounded domains. Moreover, Li and Liang \cite{16} proved  the temperature is uniformly bounded and the large-time behavior of the solutions,   the cut-off function technique and the new weighted estimation method are used in \cite{16}.

However, it is much more complicated to obtain the  existence and large-time behavior of solutions for the system \eqref{01}-\eqref{03} and \eqref{23} with different end states, i.e. $(v_{+},u_{+},\theta_{+})\neq
(v_{-},u_{-},\theta_{-})$. The large-time behavior of solutions to \eqref{01}-\eqref{03} is closely related to the Riemann problem of the associated Euler equations:
\begin{equation}
\left\{ \begin{array}{l}
\displaystyle  v_{t}-u_{x}=0,\\ 
\displaystyle  u_{t}+p_{x}=0,\\
\displaystyle  \big(e+\frac{u^{2}}{2}\big)_{t}+(pu)_{x}=0,\\
\label{14}
\displaystyle (v,u,\theta)(x,0)=\left\{ \begin{array}{l}
(v_{-},u_{-},\theta_{-}),\quad x<0,\\ 
(v_{+},u_{+},\theta_{+}),\quad x>0, \end{array}\right. \end{array} \right.\\
\end{equation}
which is the most important hyperbolic system of conservation law. It is well known that the system \eqref{14} has three basic wave patterns (see Smoller \cite{smoller}), including two nonlinear waves such as shock and rarefaction wave, and a linearly degenerate wave called as contact discontinuity. 
The solutions consist of these  three wave patterns and their superpositions, called Riemann solutions of \eqref{14}, and govern both the local and large time asymptotic behavior of general solutions of the system \eqref{01}-\eqref{03}. Therefore, the research on the large-time behavior of the viscous version of these basic wave patterns and their superpositions to the compressible Navier-Stokes system \eqref{01}-\eqref{03}  is  important and challenging.

A large number of literature have been  devoted to the stability analysis of the viscous wave pattern to system \eqref{01}-\eqref{03}, such as, Huang-Matsumura-Shi \cite{HMS2004} presented the  asymptotic behavior for contact discontinuity of the compressible Navier-Stokes equations by constructing a nonlinear diffusive wave called as viscous contact discontinuity. Huang-Matsumura \cite{5c} proved that the stability of a composite wave of two viscous shock waves for the full compressible Navier-Stokes equation,  they removed the zero initial mass condition.  Huang-Li-Matsumura \cite{HLM2010} showed the combination  of a viscous contact wave with rarefaction waves is asymptotically stable by new estimates on the heat kernel. For more detailed work on the asymptotic stability of waves, for example,  see Matsumura-Nishihara \cite{opq}, Kawashima-Matsumura \cite{12}, Huang-Matsumura \cite{5c} for more results on shock waves,  see Matsumura-Nishihara \cite{mats}-\cite{hny}, Nishihara-Yang-Zhao \cite{Tyang} for more results on rarefaction waves results, and see Huang-Matsumura-Shi \cite{HMS2004}, Huang-Matsumura-Xin \cite{5e}, Huang-Xin-Yang \cite{5f}, Huang-Zhao \cite{5h}, Hong \cite{5} for more results on contact discontinuity waves, and so on.

However, it should be mentioned here that all results mentioned above are concerned with small perturbation around the viscous wave pattern. Thus, there results concern "local" stability. Nishihara-Yang-Zhao \cite{Tyang} first considered the stability on rarefaction waves for system \eqref{01}-\eqref{03} with "partially" large perturbation. In \cite{Tyang}, they considered the condition that the adiabatic exponent $\gamma$ is closing to $1$, where the perturbation around the rarefaction wave can be larger if $\gamma-1$ is smaller. Huang-Zhao \cite{5h} extended this result to that a single viscous contact wave and the combination of viscous contact wave and rarefaction waves for a free boundary value problem. The condition on $\gamma$ is crucial in Huang-Zhao \cite{5h} and Nishilara-Yang-Zhao \cite{Tyang}, but it is not natural in the physical setting. Recently, Huang-Wang \cite{5bb} proved that stability of superposition of viscous contact wave and rarefaction waves under large initial perturbation without any restriction on $\gamma$, provided the strength of the wave patterns $\delta$ is small enough.

Before stating the main results, we shall first recall the viscous contact wave $(V,U,\Theta)$ for the compressible Navier-Stokes system \eqref{01}-\eqref{03} introduced in Huang-Matsumura-Shi \cite{HMS2004}. For the Riemann problem \eqref{14}, it is known that there exists the unique entropy solution $(v^c,u^c,\theta^c)$ called as the contact discontinuity of the form
\begin{equation}
(v^c,u^c,\theta^c)(x,t)=\left\{ \begin{array}{l}
(v_{-},u_{-},\theta_{-}),\quad x<0, t>0,\\ 
(v_{+},u_{+},\theta_{+}),\quad x>0, t>0,\label{16}
       \end{array}\right.
\end{equation}
provided that 
\begin{gather}
    u_{-}=u_{+},\quad p_{-}\overset{\text{def}}{=}\frac{R\theta_{-}}{v_{-}}=p_{+}\overset{\text{def}}{=}\frac{R\theta_{+}}{v_{+}}.\label{71}
\end{gather}
Without loss of generality,  we assume that $u_{-}=u_{+}=0$. Thanks to the diffusion wave constructed in Huang-Matsumura-Shi \cite{HMS2004}, the viscous contact discontinuous wave $(V,U,\Theta)$ is introduced as following:
 \begin{gather}
    V=\frac{R}{p_{+}}\Theta,\quad U=\frac{\kappa(\gamma-1)}{\gamma R}\frac{\Theta_{x}}{\Theta},
    \quad \Theta=\Theta\Big(\frac x{\sqrt{1+t}}\Big),\label{18}
\end{gather}
where $\Theta$ is the self-similarity solution of the following heat equation
\begin{equation}\label{931}
\left\{ \begin{array}{l}  \displaystyle \Theta_{t}=\frac{\kappa p_{+}(\gamma-1)}{\gamma R^{2}}\big( \frac{\Theta_{x}}{V} \big)_{x}, \\
   \displaystyle  \lim_{x\rightarrow\pm \infty}\Theta(x,t)=\theta_{\pm},
\end{array}\right. \end{equation}
which has a unique self-similar solution $\Theta(x,t)=\Theta({\xi})$, $\xi=x/\sqrt{1+t}$ duo to Hsiao-Liu \cite{15}.
Furthermore, $\Theta(\xi)$ is a monotone function, increasing if $\theta_{+}>\theta_{-}$ and decreasing if $\theta_{+}<\theta_{-}$. On the other hand, there exists some positive constant $\delta$ such that, for $\delta=\mid \theta_{+}-\theta_{-} \mid$, $\Theta$ satisfies
\begin{gather}
    (1+t)| \Theta_{xx}|+(1+t)^{\frac12}| \Theta_{x}|+|\Theta-\theta_{\pm}|
    =O(1)\delta e^{-\frac{c_{1}x^{2}}{1+t}},\label{1.12}
\end{gather}
as $| x| \to \infty$, where $c_{1}$ is a positive constant depending only on $\theta_{\pm}$. Moreover, $(V,U,\Theta)$ satisfies
\begin{equation}
\left\{ \begin{array}{l}
\displaystyle V_{t}-U_{x}=0,\\ 
\displaystyle U_{t}+\big( \frac{R\Theta}{V} \big)_{x}=\mu\big(\frac{U_{x}}{V}\big)_{x}+R_{1},\\
\displaystyle c_{\nu}\Theta_{t}+p(V,\Theta)U_{x}=\kappa\big( \frac{\Theta_{x}}{V} \big)_{x}+\mu\frac{U_{x}^{2}}{V}+R_{2},
\label{22}
       \end{array} \right.\\
\end{equation}
where 
\begin{gather}
    R_{1}=U_{t}-\mu\big( \frac{U_{x}}{V} \big),\quad R_{2}=-\mu\frac{U_{x}^{2}}{V}.
\end{gather}
 Now, the perturbation near the viscous contact discontinuity wave $(V,U,\Theta)$ is represented by the following function  $(\phi,\psi,\zeta)(x,t)$ below 
  \begin{gather}
     (\phi, \psi, \zeta)(x,t)=(v-V, u-U, \theta-\Theta)(x,t).
 \end{gather}
 System \eqref{01}--\eqref{03} is supplemented with the following initial conditions
 \begin{gather}
     \big( \phi_{0}(x),~\psi_{0}(x),~ \zeta_{0}(x) \big)\in H^{1}(\mathbb{R}),\quad
     \inf_{x\in\mathbb{R}}v_{0}>0,~ \inf_{x\in\mathbb{R}}\theta_{0}>0.\label{98}
 \end{gather}
Now, the main results of this paper is given as following: 
\begin{theorem}\label{70}
(Viscous contact wave). Under the assumption of the initial condition \eqref{98}, for any given left end state $(v_{-},u_{-},\theta_{-})$, suppose that the right end state $(v_{+},u_{+},\theta_{+})$ satisfies \eqref{71}. Let $(V,U,\Theta)$ be the viscous contact wave defined in \eqref{18} with strength $\delta=\mid \theta_{+}-\theta_{-} \mid$. There exists a small constant $\delta_{0}$ depending on 
\begin{gather}
  \inf_{x\in \mathbb{R}}v_{0},~\inf_{x\in\mathbb{R}}\theta_{0},~\|\phi_{0}\|_{H^{1}(\mathbb{R})}, ~\|\psi_{0}\|_{L^{2}(\mathbb{R})},~\|\psi_{0}\|_{L^{4}(\mathbb{R})},~\|\zeta_{0}\|_{L^{2}(\mathbb{R})},\label{strong}
\end{gather}
such that if $\delta<\delta_{0}$,
 the Cauchy problem \eqref{01}--\eqref{03}, \eqref{23} admitting  a unique global solution $(v,u,\theta)$ satisfying 
\begin{gather}
    (v-V,~u-U,~\theta-\Theta)\in C((0,+\infty);H^{1}(\mathbb{R})),\notag\\ 
    (v-V)_{x}\in L^{2}(0,+\infty;L^{2}(\mathbb{R})),\notag \\
    (u-U,~\theta-\Theta)_{x}\in L^{2}(0,+\infty;H^{1}(\mathbb{R})).\notag
\end{gather}
Furthermore, 
\begin{eqnarray}
 &&   0<\inf_{(x,t)\in\mathbb{R}\times(0,+\infty)}v(x,t)\leq \sup_{(x,t)\in\mathbb{R}\times(0,+\infty)}v(x,t)<+\infty,\notag\\
 &&   0<\inf_{(x,t)\in\mathbb{R}\times(0,+\infty)}\theta(x,t)\leq \sup_{(x,t)\in\mathbb{R}\times(0,+\infty)}\theta(x,t)<+\infty,\notag
\end{eqnarray}
and 
\begin{gather}
   \lim_{t\to+\infty} \sup_{x\in\mathbb{R}}\big|(v-V,~u-U,~\theta-\Theta)(x,t)\big|=0.\label{large1}
\end{gather}
\end{theorem}

\begin{remark}
The condition \eqref{strong} shows that the initial perturbation  $\|\phi_{0}\|_{H^{1}(\mathbb{R})}$, $\|\psi_{0}\|_{L^{2}(\mathbb{R})}$,
$\|\psi_{0}\|_{L^{4}(\mathbb{R})}$, $\|\zeta_{0}\|_{L^{2}(\mathbb{R})}$ can be lager when the strength of the contact wave is smaller. Additionally, the initial perturbation  $\|\phi_{0x}\|_{L^{2}(\mathbb{R})}$ and $\|\zeta_{0x}\|_{L^{2}(\mathbb{R})}$ can be arbitrarily large.
\end{remark}

\begin{remark}
Our Theorem 1.1 improves Huang-Wang's result \cite{5bb} where they need the  $\delta_{0}$  depending on $\displaystyle\inf_{x\in\mathbb{R}}v_{0},~\displaystyle\inf_{x\in\mathbb{R}}\theta_{0}$ and $\|(\phi_{0},\psi_{0},\zeta_{0})\|_{H^{1}(\mathbb{R})}$ which is indeed stronger than \eqref{strong}.
\end{remark}

\vskip 0.3cm
When the relation \eqref{71} fails, the basic theory of hyperbolic systems of conservation laws implies that for any given constant state $(v_{-},u_{-},\theta_{-})$ with $v_{-}>0,~\theta_{-}>0$ and $u_{-}\in\mathbb{R}$, there exists a suitable neighborhood $O_{(v_{-},u_{-},\theta_{-})}$ of $(v_{-},u_{-},\theta_{-})$ such that for any $(v_{+},u_{+},\theta_{+})\in O_{(v_{-},u_{-},\theta_{-})}$, the Riemann problem of the Euler system \eqref{14} has a unique solution. In this paper, we only consider the stability of the superposition of the viscous contact wave and rarefaction wave. In this situation, we assume that 
\begin{gather}
    (v_{+},u_{+},\theta_{+})\in R_{1}CR_{3}(v_{-},u_{-},\theta_{-})\subset O_{(v_{-},u_{-},\theta_{-})},\notag
\end{gather}
where
\begin{eqnarray}
    R_{1}CR_{3}(v_{-},u_{-},\theta_{-})&\overset{\text{def}}{=}&\Big\{(v,u,\theta)\in O_{(v_{-},u_{-},\theta_{-})}\Big|u\geq u_{-}-\int_{v_{-}}^{ve^{\frac{(\gamma-1)(s_{-}-s)}{R\gamma}}}\lambda_{-}(\eta,s_{-})d\eta,\notag\\
        &&\qquad\qquad u\geq u_{-}-\int_{v_{-}e^{\frac{(\gamma-1)(s-s_{-})}{R\gamma}}}^{v}\lambda_{+}(\eta,s)d\eta,s\neq s_{-}        \Big\},\label{RCR}
  \end{eqnarray}
with
\begin{gather}
    s=\frac{R}{\gamma-1}\ln \frac{R\theta}{A}+R\ln v,~~s_{\pm}=\frac{R}{\gamma-1}\ln \frac{R\theta_{\pm}}{A}+R\ln v_{\pm}, \ \lambda_{\pm}(v,s)=\pm\sqrt{A\gamma v^{-\gamma-1}e^{\frac{(\gamma-1)s}R}}.\notag
\end{gather}
By the standard argument of Smoller \cite{smoller}, there exists some suitably small $\delta_{1}>0$, such that for 
\begin{gather}
    (v_{+},u_{+},\theta_{+})\in R_{1}CR_{3}(v_{-},u_{-},\theta_{-}),\quad |\theta_{+}-\theta_{-}|\leq\delta_{1},\label{1.18}
\end{gather}
there exists a positive constant $C=C(\theta_{-},\delta_{1})$ and a unique pair of points $(v_{-}^{m},u^{m},\theta_{-}^{m})$ and $(v_{+}^{m},u^{m},\theta_{+}^{m})$ in $O_{(v_{-},u_{-},\theta_{-})}$ satisfying
\begin{gather}
    \frac{R\theta_{-}^{m}}{v_{-}^{m}}=\frac{R\theta_{+}^{m}}{v_{+}^{m}}\overset{\text{def}}{=} p^{m},\notag
\end{gather}
and 
\begin{gather}
   |v_{\pm}^{m}-v_{\pm}|+|u^{m}-u_{\pm}|+|\theta_{\pm}^{m}-\theta_{\pm}|\leq C|\theta_{+}-\theta_{-}|.
\end{gather}
Moreover, the points $(v_{-}^{m},u^{m},\theta_{-}^{m})$ and $(v_{+}^{m},u^{m},\theta_{+}^{m})$ belong to the 1-rarefaction wave curve $R_{-}(v_{-},u_{-},\theta_{-})$ and the 3-rarefaction wave curve $R_{+}(v_{+},u_{+},\theta_{+})$, where
\begin{gather}
    R_{\pm}(v_{\pm},u_{\pm},\theta_{\pm})=\Big\{ (v,u,\theta)\Big|s=s_{\pm},u=u_{\pm}-\int_{v_{\pm}}^{v}\lambda_{\pm}(\eta,s_{\pm})d\eta,v>v_{\pm} \Big\}.\notag
\end{gather}
Without loss of generality, we assume $u^{m}=0$,  then, the 1-rarefaction wave $(v_{-}^{r},u_{-}^{r},\theta_{-}^{r})(x/t)$  connecting $(v_{-},u_{-},\theta_{-})$ and $(v_{-}^{m},0,\theta_{-}^{m})$ and the 3-rarefaction wave $(v_{+}^{r},u_{+}^{r},\theta_{+}^{r})(x/t)$ connecting $(v_{+}^{m},0,\theta_{+}^{m})$ and $(v_{+},u_{+},\theta_{+})$ are the weak solutions of the Riemann problem of the Euler system \eqref{14} with the following initial Riemann datas respectively:
\begin{equation}
(v_{\pm}^{r},u_{\pm}^{r},\theta_{\pm}^{r})(x,0)=(v_{\pm}^{m},0,\theta_{\pm}^{m}),\ \pm x<0;\qquad (v_{\pm}^{r},u_{\pm}^{r},\theta_{\pm}^{r})(x,0)=
(v_{\pm},u_{\pm},\theta_{\pm}),\ \pm x>0.\label{9658}
\end{equation}
Since the rarefaction wave $(v_{\pm}^{r},u_{\pm}^{r},\theta_{\pm}^{r})$  is not differentiable, it is necessary to construct a smooth approximation to the function in order to obtain the asymptotic behavior for the solutions of \eqref{01}-\eqref{23}. 
 Motivated by \cite{mats}, the smooth solutions of Euler system  \eqref{14}, $(V_{\pm}^{r},U_{\pm}^{r},\Theta_{\pm}^{r})$, which approximate $(v_{\pm}^{r},u_{\pm}^{r},\theta_{\pm}^{r})$, are given by
\begin{equation}
\left\{ \begin{array}{l}
\displaystyle\lambda_{\pm}(V_{\pm}^{r}(x,t),s_{\pm})=w_{\pm}(x,t),\\ 
\displaystyle U_{\pm}^{r}=u_{\pm}-\int_{v_{\pm}}^{V_{\pm}^{r}(x,t)}\lambda_{\pm}(\eta,s_{\pm})d\eta,\\
\displaystyle\Theta_{\pm}^{r}=\theta_{\pm}(v_{\pm})^{\gamma-1}(V_{\pm}^{r})^{1-\gamma},
       \end{array} \right.\label{1.20}\\
\end{equation}
where $w_{-}$ (respectively $w_{+}$) is the solution of the initial problem for the typical Burgers equation:
\begin{equation}
\left\{ \begin{array}{l}
\displaystyle w_{t}+ww_{x}=0,\quad (x,t)\in\mathbb{R}\times(0,\infty),\\ 
\displaystyle w(x,0)=(\frac{w_{r}+w_{l}}2+\frac{w_{r}-w_{l}}2)\tanh x,\label{tanh}
       \end{array} \right.\\
\end{equation}
with $w_{l}=\lambda_{-}(v_{-},s_{-})$, $w_{r}=\lambda_{-}(v_{-}^{m},s_{-})$, (respectively $w_{l}=\lambda_{+}(v_{+}^{m},s_{+})$, $w_{r}=\lambda_{+}(v_{+}^{m},s_{+})$).

Let $(V^{cd},U^{cd},\Theta^{cd})(x,t)$ be the viscous contact wave constructed by \eqref{18} which satisfies \eqref{931}, with $(v_{\pm},u_{\pm},\theta_{\pm})$ replaced by $(v_{\pm}^{m},0,\theta_{\pm}^{m})$.
Without causing ambiguity, for simplicity's sake, we will still use $(V,U,\Theta)$ to the superposition of the smooth approximate 1-rarefaction wave, viscous contact discontinuity wave and the smooth approximate 2-rarefaction wave, as follows
\begin{eqnarray}
\displaystyle \big(V,U,\Theta\big)(x,t)&=&\big(V_-^r,U_-^r,\Theta_-^r\big)(x,t)+\big(V^{cd},U^{cd},\Theta^{cd}\big)(\frac x{\sqrt{1+t}})+\big(V_+^r,U_+^r,\Theta_+^r\big)(x,t)\notag\\
\displaystyle&&\qquad-\big(v_-^m,0,\theta_-^m\big)-\big(v_+^m,0,\theta_+^m\big),\label{V25}
\end{eqnarray}
and 
\begin{gather}
    (\phi,\psi,\zeta)(x,t)=(v-V,u-U,\theta-\Theta)(x,t).\notag
\end{gather}
So far, another main theorem in this paper is described as follows:
\begin{theorem}\label{4925}
(Composite wave). Under the assumption of the initial condition \eqref{98}, for any given left end state $(v_{-},u_{-},\theta_{-})$, let $(V,U,\Theta)$ be defined in \eqref{V25} with strength $\delta=\mid \theta_{+}-\theta_{-} \mid$. There exists a small constant $\delta_{0}$ depending on 
\begin{gather}
  \inf_{x\in \mathbb{R}}v_{0},~\inf_{x\in\mathbb{R}}\theta_{0},~\|\phi_{0}\|_{H^{1}(\mathbb{R})}, ~\|\psi_{0}\|_{L^{2}(\mathbb{R})},~\|\psi_{0}\|_{L^{4}(\mathbb{R})},~\|\zeta_{0}\|_{L^{2}(\mathbb{R})},\label{composite}
\end{gather}
such that if $\delta<\delta_{0}$, the Cauchy problem \eqref{01}--\eqref{03}, \eqref{23}  admits  a unique global solution $(v,~u,~\theta)$ satisfying 
\begin{gather}
    (v-V,~u-U,~\theta-\Theta)\in C((0,+\infty);H^{1}(\mathbb{R})),\notag\\ 
    (v-V)_{x}\in L^{2}(0,+\infty;L^{2}(\mathbb{R})),\notag \\
    (u-U,~\theta-\Theta)_{x}\in L^{2}(0,+\infty;H^{1}(\mathbb{R})).\notag
\end{gather}
and large time behavior
\begin{gather}
   \lim_{t\to+\infty} \sup_{x\in\mathbb{R}}\big|(v-V,~u-U,~\theta-\Theta)(x,t)\big|=0.\label{large}
\end{gather}
where the $(v_{-}^{r},u_{-}^{r},\theta_{-}^{r})(x,t)$ and $(v_{+}^{r},u_{+}^{r},\theta_{+}^{r})(x,t)$ are the 1-rarefaction and 3-rarefaction waves uniquely determined by \eqref{14} and \eqref{9658}. 
\end{theorem}

\begin{remark}
The condition \eqref{composite} means that if the strength of the composite wave is smaller, the initial perturbation $\|\phi_{0}\|_{H^{1}(\mathbb{R})}$, $\|\psi_{0}\|_{L^{2}(\mathbb{R})}$, $\|\psi_ {0}\|_{L^{4}(\mathbb{R})}$, $\|\zeta_{0}\|_{L^{2}(\mathbb{R})}$ can be larger.
\end{remark} 

\begin{remark}
Theorem 1.2 implies that the global solution $(v,u,\theta)$ of system \eqref{01} has following large time behavior
\[ \lim_{t\to +\infty}\sup_{x\in\mathbb{R}} \left(
\begin{array}{cccc}
|(v-v_{-}^{r}-V^{cd}-v_{+}^{r}+v_{-}^{m}+v_{+}^{m})(x,t)|\\
|(u-u_{-}^{r}-U^{cd}-u_{+}^{r})(x,t)|\\
|(\theta-\theta_{-}^{r}-\Theta^{cd}-\theta_{+}^{r}+\theta_{-}^{m}+\theta_{+}^{m})(x,t)|\\
\end{array} \right)=0. \]
\end{remark}

\vskip 0.3cm
\noindent\textbf{\normalsize Notations.} We denote by $C$ and $c$ the positive generic constants  without confusion throughout this paper, and the symbol $\int_{\mathbb{R}}\cdot dx$ is abbreviated as $\int\cdot dx$.
  $L^2(\mathbb{R})$ denotes the space of Lebesgue measurable functions on $\mathbb{R}$ which are square integrable, with the norm $\|f\|=(\int_{\mathbb{R}}|f|^2)^{\frac{1}{2}}$.
 $H^l(\mathbb{R})(l\geq0)$ denotes the Sobolev space of $L^2$-functions $f$ on $\mathbb{R}$ whose derivatives $\partial^j_x f,  j=1,\cdots$ are $L^2$
 functions too, with the norm
$ \|f\|_{H^l(\mathbb{R})}=(\sum_{j=0}^l\|\partial^j_x f\|^2)^{\frac{1}{2}}$.

\

We now make some comments on the analysis of this paper. First, we consider the stability of the wave consisting of the viscous contact one. Motivated by Huang-Wang\cite{5bb}, Jiang \cite{9}, Kazhikhov-Shelukhin \cite{15a} and Li-Liang \cite{16}, we first assume that the energy estimate (see  \eqref{25}) is bounded by $2C_{0}$ (see \eqref{C0}). Second, much as in Jiang \cite{9,10}, a local representation of the specific volum $v$ can be derived by using a special cut-off function when the far field states are different. Thanks to the choice of the cut-off function and delicate analysis based on the \eqref{25}, the specific volume $v$ is shown uniformly bounded from below and above with respect to the space and time(see Lemma 3.2).  Next, motivated by Li-Liang \cite{16} and Huang-Wang \cite{5bb}, we multiply the temperature  and momentum equation by $(\zeta-\Theta)_{+}$ (see \eqref{91}) and $2\psi(\zeta-\Theta)_{+}$ (\eqref{35}), and adding them together, we can control the two terms $\|\zeta_{x}\|_{L^{2}(\mathbb{R}\times(0,\infty))}$ and $\|\sqrt{\theta}\psi_{x}\|_{L^{2}(\mathbb{R}\times(0,\infty))}$. Then,
the key step  is to reduce the depending of $\delta_0$. To overcome this problem, the cut-off function $\tilde\sigma$ is introduced as \eqref{sigma}, and with the help of this cut-off function $\tilde{\sigma}$,   multiplying the momentum and the energy equation by $-\tilde{\sigma}\psi_{xx}$ and $-\tilde{\sigma}^{2}\zeta_{xx}$, after some laborious estimates, we can get that the  $\delta_{0}$ is only depended on  
\begin{gather}
  \inf_{x\in \mathbb{R}}v_{0},~\inf_{x\in\mathbb{R}}\theta_{0},~\|\phi_{0}\|_{H^{1}(\mathbb{R})}, ~\|\psi_{0}\|_{L^{2}(\mathbb{R})},~\|\psi_{0}\|_{L^{4}(\mathbb{R})},~\|\zeta_{0}\|_{L^{2}(\mathbb{R})}. \notag
\end{gather}
Finally, we can close the assuming \eqref{25} by above estimate. We remark that the underlying structures of viscous contact wave and rarefaction waves are essentially used throughout the whole proof.

The rest of this paper is organized as follows. Some facts and elementary properties of viscous contact wave and rarefaction waves are collected in Section 2. Section 3 and Section 4 is devoted to proving the Theorem 1.1 and Theorem 1.2.

\section{Preliminaries}
\indent\qquad
In this section, we recall some known results about the  properties of  viscous contact wave $(V,U,\Theta)$ defined by  \eqref{18}, which is useful in the following sections.
\begin{lemma}
Assume that $\delta=\mid\theta_{+}-\theta_{-}\mid \leq\delta_{0}$ for a small positive constant $\delta_{0}$. Then, the viscous contact wave $V,U,\Theta$ defined by \eqref{18} has the following properties:
\begin{equation}
    \begin{split}
        &\mid V-v_{\pm}\mid +\mid \Theta-\theta_{\pm}\mid\leq O(1)\delta e^{-\frac{c_{1}x^{2}}{1+t}},\\
        &\mid\partial_{x}^{k}V\mid+\mid\partial_{x}^{k-1}U\mid +\mid \partial_{x}^{k}\Theta\mid
        \leq O(1)\delta(1+t)^{-\frac k2}e^{-\frac{c_{1}x^{2}}{1+t}},\quad k\geq 1.\label{41}
    \end{split}
\end{equation}
\end{lemma}

The proof of Lemma 2.1 is given by  Huang-Matsumura-Shi \cite{HMS2004} and  will not be repeated here. By using \eqref{41}, the following estimate are obtained directly:
\begin{equation}\label{43}
    \begin{split}
        &R_{1}=O(1)\delta(1+t)^{-\frac32}e^{-\frac{c_{1}x^{2}}{1+t}},\quad R_{2}=O(1)\delta(1+t)^{-2}e^{-\frac{c_{1}x^{2}}{1+t}}.
    \end{split}
\end{equation}
The following two Lemmas are  important for  basic energy estimate, the proof can be found in Huang-Li-Matsumura \cite{HLM2010} and the proof here is omitted for the sake of simplicity.

\begin{lemma}\label{20}
For $0<T\leq +\infty$, suppose that $h(x,t)$ satisfies
\begin{gather}
    h\in L^{\infty}(0,T;L^{2}(\mathbb{R})),\quad h_{x}\in L^{2}(0,T;L^{2}(\mathbb{R})),\quad h_{t}\in 
    L^{2}(0,T;H^{-1}(\mathbb{R})).
\end{gather}
Then, for $\alpha>0$, it holds that
\begin{equation}
    \begin{split}
        &\int_{0}^{T}\int h^{2}w^{2}dxdt\leq 4\pi\| h(0) \|^{2}+4\pi\alpha^{-1}\int_{0}^{T}\| h_{x} \|^{2}dt+
        8\alpha\int_{0}^{T}<h_{t},hg^{2}>_{H^{-1}\times H^{2}}dt,\notag
    \end{split}
\end{equation}
where  
\begin{gather}
w(x,t)=(1+t)^{-1/2}\exp{\big( -\frac{\alpha x^{2}}{1+t} \big)},\quad g(x,t)=\int_{-\infty}^{x}w(y,t)dy.\notag    
\end{gather}
\end{lemma}

\begin{lemma}\label{21}
For $\alpha\in(0,\frac{c_{1}}4]$ and $w$ defined in Lemma \ref{20}, there exists some positive constant $C$ depending on $\alpha$, such that the following estimate holds:
\begin{gather}
    \int_{0}^{t}\int\big( \phi^{2}+\psi^{2}+\zeta^{2} \big)w^{2}dxds\leq C\Big( 1+\int_{0}^{t}\int\big( \phi_{x}^{2}+\psi_{x}^{2}+\zeta_{x}^{2} \big)dxds \Big).
\end{gather}
\end{lemma}

\vskip 0.3cm
Next,  the following properties of the smooth approximate rarefaction waves for the system \eqref{tanh} are given below, see  Matsumura-Nishihara \cite{mats}.
\begin{lemma}
For given $w_{l}\in\mathbb{R}$ and $\bar{w}>0$, let $w_{r}\in\{ 0<\hat{w}\overset{\mathrm{def}}{=} w-w_{l}<\bar{w} \}$. Then, the problem \eqref{tanh} has a unique smooth solution in time satisfying the following properties:

(i)~$w_{l}<w(x,t)<w_{r},w_{x}>0,  \forall x\in\mathbb{R},t>0.$

(ii)~For $p\in[1,\infty]$, there exists some positive constant $C=C(p,w_{l},\bar{w})$ such that, for $\hat{w}\geq 0$ and $t\geq 0$,
\begin{gather}
    \|w_{x}(t)\|_{L^{p}(\mathbb{R})}\leq C\min\big\{\hat{w},\hat{w}^{\frac1p}t-1+\frac1p\big\},\quad \|w_{xx}\|_{L^{p}(\mathbb{R})}\leq C\min\big\{\hat{w},t^{-1}\big\}.\notag
\end{gather}

(iii)~If $w_{l}>0$, for any $(x,t)\in(-\infty,0]\times[0,+\infty)$,
\begin{gather}
    |w(x,t)-w_{l}|\leq \hat{w}e^{-2(|x|+w_{l}t)},\quad |w_{x}(x,t)|\leq 2\hat{w}e^{-2(|x|+w_{l}t)}.\notag
\end{gather}

(iv)~If $w_{l}<0$, for any $(x,t)\in[0,\infty)\times[0,\infty)$,
\begin{gather}
    |w(x,t)-w_{r}|\leq \hat{w}e^{-2(|x|+w_{r}t)},\quad |w_{x}(x,t)|\leq 2\hat{w}e^{-2(|x|+w_{r}t)}.\notag
\end{gather}

(v)~For the Riemann solution $w^{r}(x/t)$ of the scalar equation \eqref{tanh} with the Riemann initial data
\begin{equation}
w(x,0)=\left\{ \begin{array}{l}
w_{l},\quad x<0,\\ 
w_{r},\quad x>0,
       \end{array}\right.\notag
\end{equation}
we have 
\begin{gather}
    \lim_{t\to+\infty}\sup_{x\in\mathbb{R}}|w(x,t)-w^{r}(\frac xt)|=0.\notag
\end{gather}
\end{lemma}

Finally, we divide $\mathbb{R}\times(0,t)$ into three parts, that is, $\mathbb{R}\times(0,t)=\Omega_{-}\cup\Omega_{c}\cup\Omega_{+}$, with
\begin{eqnarray}
 &&   \Omega_{\pm}=\Big\{(x,t)\big|\pm2x>\pm\lambda_{\pm}(v_{\pm}^{m},s_{\pm})t\Big\},\\
 &&   \Omega_{c}=\Big\{(x,t)\big|\lambda_{-}(v_{-}^{m},s_{-})t\leq 2x\leq\lambda_{+}(v_{+}^{m},s_{+})t \Big\}.
\end{eqnarray}
Then, from Lemma 2.4 and \eqref{1.12},  the following results hold.

\begin{lemma}
For any given $(v_{-},u_{-},\theta_{-})$, assume that $(v_{+},u_{+},\theta_{+})$ satisfies \eqref{1.18} with 
$\delta=|\theta_{+}-\theta_{-}| \leq \bar{\delta}$. Then the smooth rarefaction waves $(V_{\pm}^{r},U_{\pm}^{r},\Theta_{\pm}^{r})$ constructed in \eqref{1.20} and the viscous contact discontinuity wave $V^{cd},U^{cd},\Theta^{cd}$ satisfy the following:

(i)~$U_{\pm}^{r}\geq 0,\quad (x\in \mathbb{R},t>0).$

(ii)~For $1\leq p\leq\infty$, there exists a positive constant $C=C(p,v_{-},u_{-},\theta_{-},\bar{\delta},\delta_{1})$ such that for $\delta=|\theta_{+}-\theta_{-}|$ and $t\geq 0$, it holds that
\begin{equation}
\left. \begin{array}{l}
 \displaystyle   \ \big\|\big( (V_{\pm}^{r})_{x},(U_{\pm}^{r})_{x},(\Theta_{\pm}^{r})_{x} \big)(t)\big\|_{
    L^{p}(\mathbb{R})}\leq C\min\big\{ \delta,\delta^{\frac1p}t^{-1+\frac1p} \big\},\notag\\
\displaystyle    \big\|\big( (V_{\pm}^{r})_{xx},(U_{\pm}^{r})_{xx},(\Theta_{\pm}^{r})_{xx} \big)(t)\big\|_{
    L^{p}(\mathbb{R})}\leq C\min\big\{ \delta,\delta^{-1} \big\}.\notag
\end{array} \right.
\end{equation}

(iii)~There exists some positive constant $C=C(v_{-},u_{-},\theta_{-},\bar{\delta},\delta_{1})$ such that for 
\begin{gather}
    c_{0}=\frac{1}{10}\min\big\{ |\lambda_{-}(v_{-}^{m},s_{-})|,\lambda_{+}(v_{+}^{m},s_{+}),c_{2}\lambda_{-}^{2}(v_{-}^{m},s_{-}), c_{2}\lambda_{+}^{2}(v_{+}^{m},s_{+}),1 \big\},\notag
\end{gather}
it holds that, 
\begin{gather}
    (U_{\pm}^{r})_{x}+\left| (V_{\pm}^{r})_{x} \right|+\left| V_{\pm}^{r}-v_{\pm}^{m} \right|+\left| (\Theta_{\pm}^{r})_{x} \right|+\left| \Theta_{\pm}^{r}-\theta_{\pm}^{m} \right|\leq C\delta e^{-c_{0}(|x|+t)},\ \ \mathrm{in}\  \Omega_{c},\notag
\end{gather}
and
\begin{equation}
\left\{ \begin{array}{l}
\displaystyle |V^{cd}-v_{\mp}|+|V_{x}^{cd}|+|\Theta^{cd}-\theta_{\mp}^{m}|+|\Theta_{x}^{cd}|+|U_{x}^{cd}|\leq C\delta e^{-c_{0}(|x|+t)},\\
\displaystyle (U_{\pm}^{r})_{x}+|(V_{\pm}^{r})_{x}|+|V_{\pm}^{r}-v_{\pm}^{m}|+|(\Theta_{\pm}^{r})_{x}|+|\Theta_{\pm}^{r}-\theta_{\pm}^{m}|\leq C\delta e^{-c_{0}(|x|+t)}, 
 \end{array} \right.\ \  \mathrm{in}\  \Omega_{\mp}.
\end{equation}

(iv)~For the rarefaction waves $(v_{\pm}^{r},u_{\pm}^{r},\theta_{\pm}^{r})(x/t)$ determined by \eqref{14} and \eqref{9658}, it holds
\begin{gather}
  \lim_{t\to+\infty}\sup_{x\in\mathbb{R}}|(V_{\pm}^{r},U_{\pm}^{r},\Theta_{\pm}^{r})(x,t)-(v_{\pm}^{r},u_{\pm}^{r},\theta_{\pm}^{r})(\frac xt)|=0.\notag  
\end{gather}
\end{lemma}

\section{Proof of Theorem 1.1}
\indent\qquad  In this section, we will give the asymptotic stability results for viscous contact discontinuity wave. 
Substituting \eqref{22} into \eqref{01}-\eqref{03} yields
\begin{gather}
\phi_{t}-\psi_{x}=0,\label{1}\\
\psi_{t}+(p-p_{+})_{x}=\mu\big( \frac{u_{x}}{v}-\frac{U_{x}}{V} \big)_{x}-R_{1},\label{2}\\
c_{\nu}\zeta_{t}+pu_{x}-p_{+}U_{x}=\kappa\big( \frac{\theta_{x}}{v}-\frac{\Theta_{x}}{V} \big)_{x}
+\mu\big( \frac{u_{x}^{2}}{v}-\frac{U_{x}^{2}}{V} \big)-R_{2}.\label{3}
\end{gather}
The local existence of the solution for system \eqref{1}-\eqref{3} can be obtained by the fixed point method and the local linearization method, the detail is omitted, (e.g., see Huang-Matsumura-Shi \cite{HMS2004}). Therefore, in order to extend the local solution to the global solution, the following a prior estimate need to be established.

\vskip 0.3cm
\begin{proposition}{(A priori estimates)}
Assume that the conditions of Theorem 1.1 hold, then there exists a positive constant $\delta_{0}$ such that if $\delta<\delta_{0}$, we have
\begin{gather}
    \sup_{0\leq t\leq T} \|(\phi,\psi,\zeta)\|_{H^{1}(\mathbb{R})}^{2}+\int_{0}^{T}\big(\|\phi_{x}\|^{2}+\|(\psi_{x},\zeta_{x})\|_{H^{1}(\mathbb{R})}^{2}\big)ds<M,\label{propo}
\end{gather}
where $M$ denotes a constant depending on $\mu,\kappa,R,c_{\nu},v_{\pm},u_{\pm},\theta_{\pm}$ and $\delta_{0}$.
\end{proposition}
The energy inequality \eqref{propo} and the system \eqref{1}-\eqref{3} (respectively, \eqref{3301}-\eqref{3303}) imply that
\begin{gather}
    \int_{0}^{\infty}\Big(\|(\phi_{x},\psi_{x},\zeta_{x})(t)\|^{2}+\big| \frac{d}{dt}\|(\phi_{x},\psi_{x},\zeta_{x})(t)\|^{2} \big|\Big)dt<\infty,\notag
\end{gather}
which together with \eqref{propo} and the Sobolev's inequality, easily leads to the large time behavior of the solution, that is, \eqref{large1} (respectively,\eqref{large}). 

\vskip 0.3cm
\begin{proposition}
There exists a small constant $\delta_{0}$ only depending on  $\inf_{x\in \mathbb{R}}v_{0}$, $\inf_{x\in\mathbb{R}}\theta_{0}$, $\|\phi_{0}\|_{H^{1}(\mathbb{R})}$, $\|\psi_{0}\|_{L^{2}(\mathbb{R})}$, $\|\psi_{0}\|_{L^{4}(\mathbb{R})}$, $\|\zeta_{0}\|_{L^{2}(\mathbb{R})}$, such that  
\begin{equation}\label{G(t)}
  G(t)\leq C_{0},
\end{equation}
provided that  $G(t)\leq 2C_{0} $ and $\delta<\delta_{0}$, where
\begin{eqnarray}
       G(t)&\overset{\mathrm{def}}{=}&\int_{0}^{t}\int \Big( \frac{\kappa\zeta^{2}\Theta_{x}^{2}}{v\theta^{2}\Theta}
+\frac{\kappa\Theta\phi^{2}\Theta_{x}^{2}}{v\theta^{2}V^{2}}
+\frac{\kappa|\zeta\phi|\Theta_{x}^{2}}{v\theta^{2}V}+\frac{4\mu\zeta^{2}U_{x}^{2}}{v\theta\Theta} +\frac{\mu|\zeta\phi| U_{x}^{2}}{v\theta V}+\frac{\mu\theta\phi^{2}U_{x}^{2}}{v\Theta V^{2}}\Big)dxds\notag\\
&&  +\int_{0}^{t}\int\Big( p_{+}\Phi\big( \frac{V}{v} \big)+\frac{p_{+}}{|\gamma-1|}\Phi\big( \frac{\Theta}{\theta} \big)
+\frac{|\zeta|}{\theta}|p_{+}-p| \Big)|U_{x}|dxds\notag\\
&& +\int_{0}^{t}\int \Big(|\psi R_{1}|+\big|\frac{\zeta}{\theta}R_{2}\big|\Big)dxds,\label{25}
    \end{eqnarray}
and
\begin{gather}
    C_{0}\overset{\mathrm{def}}{=} \int\Big(\frac{\psi_{0}^{2}}{2}+R\Theta_{0}\Phi\big( \frac{v_{0}}{V_{0}}\big)+c_{\nu}\Theta_{0}\Phi\big( \frac{\theta_{0}}{\Theta_{0}}  \big)\Big)dx.\label{C0}
\end{gather}
\end{proposition}

\vskip 0.3cm
Before proving Proposition 3.2, let us  give the following series of Lemmas.
\begin{lemma}
There exists two positive constants $\alpha_{1}$, $\alpha_{2}$ both depending on $R,c_{\nu},\theta_{-}$ and $C_{0}$,  such that 
\begin{gather}
\alpha_{1}\leq\int_{k}^{k+1}\bar{v}dx,\quad \int_{k}^{k+1}\bar{\theta}dx\leq \alpha_{2},\quad t\geq 0,\label{alpha1}
\end{gather}
and for each $t\geq 0$, there are points $a_{k}(t), b_{k}(t)\in [k,k+1]$ such that
\begin{gather}
\alpha_{1} \leq \bar{v}(a_{k}(t),t),\quad \bar{\theta}(b_{k}(t),t)\leq \alpha_{2},\label{32}
\end{gather}
where $k=0,\pm 1,\pm 2,...$.
\end{lemma}
\begin{proof}
Multiplying \eqref{1} by $-R\Theta(v^{-1}-V^{-1})$, \eqref{2} by $\psi$, and \eqref{3} by $\zeta\theta^{-1}$,
and then adding the resulting equations together, we have
\begin{equation}
\begin{split}
&\Big(\frac{\psi^{2}}{2}+R\Theta \Phi\big(\frac{v}{V}\big)
+c_{\nu}\Theta\Phi\big(\frac{\theta}{\Theta}\big) \Big)_{t}
+\frac{\mu\Theta\psi_{x}^{2}}{v\theta}+\frac{\kappa\Theta\zeta_{x}^{2}}{v\theta^{2}}=H_{x}+Q-\psi R_{1}-\frac{\zeta}{\theta}R_{2},\label{new1}
\end{split}
\end{equation}
with $\Phi(z)=z-\ln z-1$ for $z>0$, and
\begin{equation}
\begin{split}
&H=(p_{+}-p)\psi+\mu\big(\frac{u_{x}}{v}-\frac{U_{x}}{V} \big)\psi+\frac{\kappa\zeta}{\theta}\big(\frac{\theta_{x}}{v}-\frac{\Theta_{x}}{V} \big),\\
&Q=-p_{+}\Phi\big(\frac{V}{v}\big)U_{x}+\frac{p_{+}}{1-\gamma}\Phi\big(\frac{\Theta}{\theta}\big)U_{x}
-\mu(\frac{1}{v}-\frac{1}{V})\psi_{x}U_{x}\\
&\quad \quad +\frac{\zeta}{\theta}(p_{+}-p)U_{x}+\frac{\kappa\Theta_{x}}{v\theta^{2}}\zeta\zeta_{x}
+\frac{\kappa\Theta\Theta_{x}}{Vv\theta^{2}}\phi\zeta_{x}-\frac{\kappa\Theta_{x}^{2}}{Vv\theta^{2}}\zeta\phi+\frac{2\mu U_{x}}{v\theta}\zeta\psi_{x}-\frac{\mu U_{x}^{2}}{Vv\theta}\zeta\phi,\notag
\end{split}
\end{equation}
where
\begin{equation}
\begin{split}
Q\leq &\frac{\mu\Theta}{2v\theta}\psi_{x}^{2}+\frac{\kappa\Theta}{2v\theta^{2}}\zeta_{x}^{2}+\Big( p_{+}\Phi\big( \frac{V}{v} \big)+\frac{p_{+}}{|\gamma-1|}\Phi\big( \frac{\Theta}{\theta} \big)
+\frac{|\zeta|}{\theta}|p_{+}-p| \Big)|U_{x}|\\
&+\big( \frac{\kappa\zeta^{2}\Theta_{x}^{2}}{v\theta^{2}\Theta}
+\frac{\kappa\Theta\phi^{2}\Theta_{x}^{2}}{v\theta^{2}V^{2}}
+\frac{\kappa|\zeta\phi|\Theta_{x}^{2}}{v\theta^{2}V}+\frac{4\mu\zeta^{2}U_{x}^{2}}{v\theta\Theta}+\frac{\mu|\zeta\phi| U_{x}^{2}}{v\theta V}
+\frac{\mu\theta\phi^{2}U_{x}^{2}}{v\theta V^{2}} \big).\label{Q}
\end{split}
\end{equation}
Then, integrating \eqref{new1} over $\mathbb{R}\times(0,t)$ and using \eqref{Q} lead to  
\begin{eqnarray}
&&\int\Big(\frac{\psi^{2}}{2}+R\Theta\Phi\big( \frac{v}{V}\big)+c_{\nu}\Theta\Phi\big( \frac{\theta}{\Theta}  \big)\Big)dx
+ \frac{1}{2}\int_{0}^{t}\int\big( \frac{\mu\Theta\psi_{x}^{2}}{v\theta}+\frac{\kappa\Theta\zeta_{x}^{2}}{v\theta^{2}} \big)dxds\notag\\
&&\leq \int_{0}^{t}\int\Big( p_{+}\Phi\big( \frac{V}{v} \big)+\frac{p_{+}}{|\gamma-1|}\Phi
\big( \frac{\Theta}{\theta} \big)
+\frac{|\zeta|}{\theta}|p_{+}-p| \Big)|U_{x}|dxds\label{new} \\
&&\quad  +\int_{0}^{t}\int \big( \frac{\kappa\zeta^{2}\Theta_{x}^{2}}{v\theta^{2}\Theta}
+\frac{\kappa\Theta\phi^{2}\Theta_{x}^{2}}{v\theta^{2}V^{2}}
+\frac{\kappa|\zeta\phi|\Theta_{x}^{2}}{v\theta^{2}V}+\frac{4\mu\zeta^{2}U_{x}^{2}}{v\theta\Theta} +\frac{\mu|\zeta\phi| U_{x}^{2}}{v\theta V}+\frac{\mu\theta\phi^{2}U_{x}^{2}}{v\Theta V^{2}}\big)dxds\notag\\
&&\quad +\int_{0}^{t}\int \Big(|\psi R_{1}|+\big|\frac{\zeta}{\theta}R_{2}\big|\Big)dxds
+\int\Big(\frac{\psi_{0}^{2}}{2}+R\Theta_{0}\Phi\big( \frac{v_{0}}{V_{0}}\big)+c_{\nu}\Theta_{0}\Phi\big( \frac{\theta_{0}}{\Theta_{0}}  \big)\Big)dx\leq 3C_{0},\notag
\end{eqnarray}
which in particular implies
\begin{gather}
   \int_{k}^{k+1}\big(\bar{v}(x,t)-\ln\bar{v}(x,t)-1\big)dx\leq \frac{3C_{0}}{ R\theta_{-}},\qquad
   \int_{k}^{k+1}(\bar{\theta}(x,t)-\ln\bar{\theta}(x,t)-1)dx\leq \frac{3C_{0}}{ c_{\nu}\theta_{-}},\notag
\end{gather}
where $\bar{v}=v/V$, $\bar{\theta}=\theta/\Theta$, and $k=0,\pm 1,\pm 2,...$. Applying Jessen's inequality to the convex function $y-\ln y-1=\min\left\{   \frac{3C_{0}}{ R\theta_{-}}, \frac{3C_{0}}{ c_{\nu}\theta_{-}} \right\}$ yields
\begin{gather}
    \int_{k}^{k+1}\bar{v}(x,t)dx-\ln \int_{k}^{k+1}\bar{v}(x,t)dx-1\leq \min\Big\{   \frac{3C_{0}}{ R\theta_{-}}, \frac{3C_{0}}{ c_{\nu}\theta_{-}} \Big\},\notag\\
    \int_{k}^{k+1}\bar{\theta}(x,t)dx-\ln \int_{k}^{k+1}\bar{\theta}(x,t)dx-1\leq \min\Big\{   \frac{3C_{0}}{ R\theta_{-}}, \frac{3C_{0}}{ c_{\nu}\theta_{-}} \Big\},\notag
\end{gather}
which gives
\begin{gather}
\alpha_{1}\leq\int_{k}^{k+1}\bar{v}dx,\quad \int_{k}^{k+1}\bar{\theta}dx\leq \alpha_{2},\notag
\end{gather}
where $\alpha_{1}$, $\alpha_{2}$ are two positive roots of the equation 
$$y-\ln y-1=\min\left\{   \frac{3C_{0}}{ R\theta_{-}}, \frac{3C_{0}}{ c_{\nu}\theta_{-}} \right\}.$$
 Moreover, in view of the mean value theorem, for each $t\geq 0$, there exist points $a_{k}(t)$, $b_{k}(t)\in[k,k+1]$ such that
\begin{equation}
\alpha_{1} \leq \bar{v}(a_{k}(t),t),\quad \bar{\theta}(b_{k}(t),t)\leq \alpha_{2}.
\end{equation}
\end{proof}

\vskip 0.3cm
\begin{lemma}\label{thbelow}
There exists a  positive constant $C$ depending on $\displaystyle\inf_{x\in\mathbb{R}}v_{0}, \|\phi_{0}\|_{H^{1}(\mathbb{R})}, \|\psi_{0}\|_{L^{2}(\mathbb{R})}$ and $C_{0}$,  such that
\begin{gather}
C^{-1}\leq v(x,t)\leq C,\quad x\in \mathbb{R},~ t\geq 0.\label{v}
\end{gather}
\end{lemma}
\begin{proof}
Integrating \eqref{02}  over $[k,x]$, one obtains
\begin{gather}
\big( \int_{k}^{x}udy \big)_{t}=\sigma(x,t)-\sigma(k,t),\quad x\in[k,k+1],\label{42}
\end{gather}
where 
\begin{equation}\label{sigma}
  \sigma=\frac{\mu u_{x}}{v}-\frac{R\theta}{v},
\end{equation}
Integrating \eqref{42} with respect to $x$ over $[k,k+1]$, one has
\begin{gather}
\sigma(k,t)=\int_{k}^{k+1}\sigma(x,t)dx-\big( \int_{k}^{k+1}\int_{k}^{x}udydx \big)_{t}.\notag
\end{gather}
Combining this with \eqref{42}, one gets
\begin{gather}
\displaystyle v(x,t)=B_{k}(x,t)Y_{k}(t)\mathrm{exp}\big(\int_{0}^{t}\frac{\theta}{v}ds\big),\label{425}
\end{gather}
with
\begin{eqnarray}
&&\displaystyle B_{k}(x,t)=v_{0}(x)\mathrm{exp}\Big({\int_{k}^{x}(u-u_{0})dy-\int_{k}^{k+1}\int_{k}^{x}(u-u_{0})dydx}\Big),\\
&&\displaystyle Y_{k}(t)=\mathrm{exp}\Big({\int_{0}^{t}\int_{k}^{k+1}\sigma dxds}\Big).
\end{eqnarray}
Using \eqref{425}, direct computation gives
\begin{gather}
v(x,t)=B_{k}(x,t)Y_{k}(t)+\int_{0}^{t}\frac{B_{k}(x,t)Y_{k}(t)}{B_{k}(x,s)Y_{k}(s)}\theta(x,s)ds,\label{vv}
\end{gather}
by \eqref{new}, for $(x,t)\in\mathbb{R}\times[0,+\infty)$, it holds
\begin{gather}
 0<C^{-1}\leq  B_{k}(x,t)\leq C<\infty.\label{31}
\end{gather}
From now on, we always assume $\theta_{-}<\theta_{+}$ for convenience. Thus, from the properties of the viscous contact wave, we have $\theta_{-}<\Theta(x,t)<\theta_{+}$ and $v_{-}<V(x,t)<v_{+}$. For each $t\geq 0$, there exists at least one point $x_{k}(t)\in [k,k+1]$, such that
\begin{gather}
\inf_{x\in [k,k+1]}\bar{\theta}(x,t)=\bar{\theta}(x_{k}(t),t).\notag
\end{gather}
By Cauchy's inequality and using \eqref{new}, it yields that 
\begin{equation}
\begin{split}
&\Big|\int_{s}^{t}\int_{b_{k}(\tau)}^{x_{k}(\tau)}\frac{\bar{\theta}_{y}}{\bar{\theta}}(y,\tau)dyd\tau \Big|=\Big|\int_{s}^{t}\int_{b_{k}(\tau)}^{x_{k}(\tau)}\Big( \frac{\zeta_{y}}{\theta}-\frac{\zeta\Theta_{y}}{\theta\Theta} \Big)dyd\tau \Big|\\
&\leq \int_{s}^{t}\Big( \int_{k}^{k+1}\frac{\zeta_{x}^{2}}{v\theta^{2}}dx \Big)^{1/2}\Big( \int_{k}^{k+1}vdx \Big)^{1/2}d\tau+
\int_{s}^{t}\Big( \int_{k}^{k+1}\frac{\zeta^{2}\Theta_{x}^{2}}{v\theta^{2}\Theta}dx \Big)^{1/2}
\Big( \int_{k}^{k+1}vdx \Big)^{1/2}d\tau \\
&\leq C\Big(\int_{s}^{t} \int_{k}^{k+1}\frac{\zeta_{x}^{2}}{v\theta^{2}}dx d\tau\Big)^{1/2}\sqrt{t-s}+
C\Big(\int_{s}^{t}\int_{k}^{k+1}\frac{\zeta^{2}\Theta_{x}^{2}}{v\theta^{2}\Theta^{2}}dx d\tau\Big)^{1/2}\sqrt{t-s}\\
&\displaystyle\leq C\sqrt{t-s}.\notag
\end{split}
\end{equation}
Applying Jensen's inequality to the convex function $e^{x}$, one derives, for $t\geq s\geq 0$,
\begin{equation}
\begin{split}
\int_{s}^{t}&\inf_{x\in[k,k+1]}\bar{\theta}(\cdot,\tau)d\tau=\int_{s}^{t}\bar{\theta}(x_{k}(\tau),\tau)d\tau
=\int_{s}^{t}\exp{\big(\ln\bar{\theta}(x_{k}(\tau),\tau)\big)d\tau}\\
&\geq (t-s)\exp{\big( \frac{1}{t-s}\int_{s}^{t}\ln\bar{\theta}(x_{k}(\tau),\tau)d\tau \big)}\\
&=(t-s)\exp{\Big( \frac{1}{t-s}\int_{s}^{t}\big( \int_{b_{k}}^{x_{k}}\frac{\bar{\theta_{y}}}{\bar{\theta}}dy+\ln\bar{\theta}(b_{k}(\tau),\tau) \big)d\tau \Big)}\\
&\geq (t-s)\exp \Big( -\max\{|\ln\alpha_{1}|,|\ln\alpha_{2}|\}-\frac{1}{t-s}\big| \int_{s}^{t}\int_{b_{k}(\tau)}^{x_{k}(\tau)}\frac{\bar{\theta_{y}}}{\bar{\theta}}dyd\tau \big| \Big)\\
&\geq C(t-s)e^{-\frac C{\sqrt{t-s}}},\notag
\end{split}
\end{equation}
which, together with \eqref{new} and Lemma 2.1, it yields that 
\begin{equation}
\begin{split}
\int_{s}^{t}&\int_{k}^{k+1}\sigma(x,\tau)dxd\tau=\int_{s}^{t}\int_{k}^{k+1}\big( \mu\frac{\psi_{x}}{v}-R\frac{\theta}{v}+\mu\frac{U_{x}}{v} \big)(x,\tau)dxd\tau\\
&\leq C\int_{s}^{t}\int_{k}^{k+1}\frac{\psi_{x}^{2}}{v\theta}dxd\tau
-\frac{R}{2}\int_{s}^{t}\int_{k}^{k+1}\frac{\theta}{v}dxd\tau\\
&\quad +\mu\int_{s}^{t}\int_{k}^{k+1}\frac{|U_{x}|}{v}\chi_{\{ v\leq \frac{v_{-}}{2} \}}dxd\tau
 +\mu\int_{s}^{t}\int_{k}^{k+1}\frac{|U_{x}|}{v}\chi_{\{ v> \frac{v_{-}}{2} \}}dxd\tau\\
&\leq C-\frac{R}{2}\int_{s}^{t}\big(\inf_{x\in[k,k+1]}\theta\big)\big( \int_{k}^{k+1}v^{-1}dx \big)d\tau\\
&\quad +C\int_{s}^{t}\int_{k}^{k+1}\Phi\big( \frac{V}{v}\big)\big| U_{x}  \big|dxd\tau
+C\int_{s}^{t}\frac{1}{\sqrt{1+\tau}}d\tau\\
&\leq C-\frac{t-s}{C}+C \big( \sqrt{1+t}-\sqrt{1+s} \big)\\
&\leq 2C-\frac{t-s}{2C}.\label{sigma}
\end{split}
\end{equation}
It follows from the definition of $Y_{k}(t)$ and  \eqref{sigma} that
\begin{gather}
0\leq Y_{k}(t)\leq C e^{-\frac tC},\quad  \frac{Y_{k}(t)}{Y_{k}(s)}\leq Ce^{-\frac{t-s}C},\notag
\end{gather}
which combined with \eqref{vv} and \eqref{31}  gives
\begin{gather}
v(x,t)\leq C+C\int_{0}^{t}\theta(x,s)e^{-\frac{t-s}C}ds.\notag
\end{gather}
On the other hand, for any $x\in [k,k+1]$, it holds that
\begin{equation}
\begin{split}
\big|& \bar{\theta}^{\frac12}(x,t)-\bar{\theta}^{\frac12}(b_{k}(t),t) \big|\leq \int_{k}^{k+1}\frac{|\bar{\theta}_{x}|}{\bar{\theta}^{\frac{1}{2}}}dx\\
&\leq \int_{k}^{k+1}\Big(\frac{\Theta}{\theta} \Big)^{\frac12}\Big( \big|\frac{\zeta_{x}}{\Theta}\big|+\big|\frac{\zeta\Theta_{x}}{\Theta^{2}}\big| \Big)dx\leq C\int_{k}^{k+1}\frac{|\zeta_{x}|}{\sqrt{\theta}}dx+C\int_{k}^{k+1}\frac{| \zeta\Theta_{x} |}{\sqrt{\theta}}dx\\
&\leq C\Big( \int_{k}^{k+1}\frac{\zeta_{x}^{2}}{v\theta^{2}}dx \Big)^{\frac{1}{2}}
\Big( \int_{k}^{k+1}v\theta dx \Big)^{\frac{1}{2}}
+C\Big( \int_{k}^{k+1}\frac{\zeta^{2}\Theta_{x}^{2}}{v\theta^{2}}dx \Big)^{\frac{1}{2}}
\Big( \int_{k}^{k+1}v\theta dx \Big)^{\frac{1}{2}}\\
&\leq C\Big( \big( \int_{k}^{k+1}\frac{\zeta_{x}^{2}}{v\theta^{2}}dx \big)^\frac{1}{2}+
\big( \int_{k}^{k+1}\frac{\zeta^{2}\Theta_{x}^{2}}{v\theta^{2}}dx \big)^\frac{1}{2}\Big)\max_{x\in [k,k+1]}v^{\frac{1}{2}}(x,t),\notag
\end{split}
\end{equation}
and $k=0,\pm1,\pm2,...$, which, along with \eqref{new}, it leads to
\begin{equation}
\begin{split}
\frac{C}{3}&-C\big(\max_{x\in\mathbb{R}}v(x,t)\big)\int\frac{\zeta_{x}^{2}+\zeta^{2}\Theta_{x}^{2}}{v\theta^{2}}dx\leq\theta(x,t)\\
&\leq C+C\big(\max_{x\in\mathbb{R}}v(x,t)\big)\int\frac{\zeta_{x}^{2}+\zeta^{2}\Theta_{x}^{2}}{v\theta^{2}}dx
,\quad \forall x\in\mathbb{R}.\label{ta}
\end{split}
\end{equation}
Substituting \eqref{ta} into \eqref{vv} , and applying Gronwall's inequality and \eqref{32}, one has
\begin{gather}
v(x,t)\leq C,\quad x\in\mathbb{R},t\geq 0.\label{sj}
\end{gather}
Integrating \eqref{vv} over $[k,k+1]$ with respect to $x$, after using \eqref{alpha1}, one obtains
\begin{equation}
\begin{split}
v_{-}\alpha_{1}&\leq Ce^{-\frac tC}+C\int_{0}^{t}\frac{Y(t)}{Y(s)}\int_{k}^{k+1}\theta(x,s)dxds\\
&\leq Ce^{-\frac tC}+C\int_{0}^{t}\frac{Y(t)}{Y(s)}ds,\notag
\end{split}
\end{equation}
and this directly yields that
\begin{gather}
\int_{0}^{t}\frac{Y(t)}{Y(s)}ds\geq C-Ce^{-\frac tC}.\notag
\end{gather}
From \eqref{vv}, and using \eqref{new}, \eqref{ta} and \eqref{sj}, one has
\begin{equation}
\begin{split}
v(x,t)&\geq C\int_{0}^{t}\frac{Y(t)}{Y(s)}\theta(x,s)ds\\
&\geq C\int_{0}^{t}\frac{Y(t)}{Y(s)}ds-C\int_{0}^{\frac{t}{2}}\frac{Y(t)}{Y(s)}
\int\frac{\zeta_{x}^{2}+\zeta^{2}\Theta_{x}^{2}}{v\theta^{2}}dxds -C\int_{\frac{t}{2}}^{t}\frac{Y(t)}{Y(s)}\int\frac{\zeta_{x}^{2}
+\zeta^{2}\Theta_{x}^{2}}{v\theta^{2}}dxds\\
&\geq C-Ce^{-\frac tC}-Ce^{-\frac t{2C}}\int_{0}^{t}\int\frac{\zeta_{x}^{2}
+\zeta^{2}\Theta_{x}^{2}}{v\theta^{2}}dxds -C\int_{\frac{t}{2}}^{t}\int\frac{\zeta_{x}^{2}+\zeta^{2}\Theta_{x}^{2}}{v\theta^{2}}dxds\\
&\geq \frac{C}{2},\quad x\in\mathbb{R},t\geq T_{0}.\notag
\end{split}
\end{equation}
Integrating \eqref{vv} over $[k,k+1]$ with respect to $x$ and using \eqref{alpha1}, it shows
\begin{equation}
\begin{split}
\frac{\alpha_{1}v_{-}}{Y_{k}(t)}&\leq \frac{1}{Y_{k}(t)}\int_{k}^{k+1}vdx\\
&\leq C\big( 1+\int_{0}^{t}\frac{1}{Y_{k}(s)}\int_{k}^{k+1}\theta dxds \big)\\
&\leq C\big( 1+\int_{0}^{t}\frac{1}{Y_{k}(s)}ds \big),\label{33}
\end{split}
\end{equation}
which, along with the Gronwall's inequality, yields $Y_{k}(t)\geq C^{-1}(T)$. Then, from \eqref{v}, one obtains 
\begin{gather}
v(x,t)\geq B_{k}(x,t)Y_{k}(t)\geq C^{-1}(T),\quad x\in\mathbb{R}, t\in[0,T],
\end{gather}
which, together with \eqref{33}, completes the proof of Lemma 3.2.
\end{proof}

\vskip 0.3cm
\begin{lemma}\label{24}
There exists a positive constant $C$ depending on $\displaystyle\inf_{x\in \mathbb{R}}v_{0}, \|\phi_{0}\|_{H^{1}},\|\psi_{0}\|_{L^{2}}, \|\psi_{0}\|_{L^{4}},$ $\|\zeta_{0}\|_{L^{2}}$ and $C_{0}$, such that
\begin{equation}
\begin{split}
\sup_{t\in[0,T]}\int\Big(\phi_{x}^{2}+\zeta^{2}+\psi^{4}\Big)dx
+\int_{0}^{T}\int\Big(\big(\psi_{x}^{2}+\phi_{x}^{2}\big)\theta
+\psi^{2}\psi_{x}^{2}+\zeta_{x}^{2}\Big)dxdt\leq C.\label{61}
\end{split}
\end{equation}
\end{lemma}
\begin{proof}
First, for $t\geq 0$ and $a>1$, denoting
\begin{gather}
\Omega_{a}(t)\overset{\text{def}}{=}\Big\{x\in\mathbb{R}\Big| \frac{\theta}{\Theta}(x,t)>a \Big\}.\label{omega2}
\end{gather}
Taking note of the following estimation of the measure for $\Omega_{2}$
\begin{equation}
\begin{split}
2|\Omega_{2}|&\leq \sup_{t\in[0,T]}\int_{\Omega_{2}}\frac{\theta}{\Theta}dx\\
&\leq \frac{2}{1-\ln 2}\sup_{t\in[0,T]}\int_{\mathbb{R}}\Phi(\frac{\theta}{\Theta})dx\\
&\leq \frac{6C_{0}}{c_{\nu}\theta_{-}(1-\ln 2)},\label{o}
\end{split}
\end{equation}
combining with \eqref{new}, one has at once that $\Omega_{2}(t)$ is bounded. Next,  multiplying \eqref{3} by $(\zeta-\Theta)_{+}= \max\{ \zeta-\Theta,0 \}$, integrating the resultant over $\mathbb{R}\times[0,t]$, one has
\begin{eqnarray}
&&\frac{c_{\nu}}{2}\int(\zeta-\Theta)_{+}^{2}dx+\kappa\int_{0}^{t}
\int_{\Omega_{2}}\frac{\zeta_{x}^{2}}{v}dxds\notag\\
&&=\frac{c_{\nu}}{2}\int \left( \zeta_{0}(x)-\Theta(x,0) \right)_{+}^{2}dx
-\int_{0}^{t}\int\frac{R\zeta+R\Theta}{v}\psi_{x}(\zeta-\Theta)_{+}dxds\notag\\ 
&&\quad  -\int_{0}^{t}\int\frac{R\zeta-p_{+}\phi}{v}U_{x}(\zeta-\Theta)_{+}dxds+\kappa\int_{0}^{t}\int_{\Omega_{2}}\frac{\zeta_{x}\Theta_{x}}{V}dxds\label{34}\\
&&\quad  -\kappa\int_{0}^{t}\int_{\Omega_{2}}\frac{\phi\Theta_{x}^{2}}{vV}dxds+\mu\int_{0}^{t}\int\frac{\psi_{x}^{2}}{v}(\zeta-\Theta)_{+}dxds  +2\mu\int_{0}^{t}\int\frac{\psi_{x}U_{x}}{v}(\zeta-\Theta)_{+}dxds \notag\\
&&\quad-\mu\int_{0}^{t}\int\frac{\phi U_{x}^{2}}{vV}(\zeta-\Theta)_{+}dxds -\int_{0}^{t}\int R_{2}(\zeta-\Theta)_{+}dxds
-c_{\nu}\int_{0}^{t}\int(\zeta-\Theta)_{+}\partial_{t} \Theta dxds.\notag
\end{eqnarray}
Multiplying \eqref{2} by $2\psi(\zeta-\Theta)_{+},$
and integrating the resultant  over $\mathbb{R}\times[0,t]$, the following integral identity is obtained 
\begin{eqnarray}
&&\int\psi^{2}(\zeta-\Theta)_{+}dx+2\mu \int_{0}^{t}\int\frac{\psi_{x}^{2}}{v}(\zeta-\Theta)_{+}dxds\notag\\
&&=\int \psi_{0}^{2}(x)\big(\zeta_{0}(x)-\Theta_{0}(x)\big)_{+}+2\int_{0}^{t}\int\frac{R\zeta-p_{+}\phi}{v}\psi_{x}(\zeta-\Theta)_{+}dxds\notag\\
&&\quad +2\int_{0}^{t}\int_{\Omega_{2}}\frac{R\zeta-p_{+}\phi}{v}\psi\zeta_{x}dxds -2\int_{0}^{t}\sigma\int_{\Omega_{2}}\frac{R\zeta-p_{+}\phi}{v}\psi\Theta_{x}dxds\label{35}\\
&&\quad +2\mu\int_{0}^{t}\int\frac{\phi U_{x}}{vV}\psi_{x}(\zeta-\Theta)_{+}dxds -2\mu \int_{0}^{t}\int_{\Omega_{2}}\frac{\psi\psi_{x}\zeta_{x}}{v}dxds +2\mu \int_{0}^{t}\int_{\Omega_{2}}\frac{\phi\psi U_{x}}{vV}dxds\notag\\
&&\quad+2\mu \int_{0}^{t}\int_{\Omega_{2}}\frac{\psi\psi_{x}\Theta_{x}}{v}dxds -2\mu \int_{0}^{t}\int_{\Omega_{2}}\frac{\phi\psi}{vV}U_{x}\Theta_{x}dxds
-2\int_{0}^{t}\int \psi R_{1}(\zeta-\Theta)_{+}dxds\notag\\
&&\quad  +\int_{0}^{t}\int_{\Omega_{2}}\psi^{2}\partial_{t}\zeta dxds-\int_{0}^{t}\int_{\Omega_{2}}\psi^{2}\partial_{t}\Theta dxds.\notag
\end{eqnarray}
Adding \eqref{35} into \eqref{34},  it shows that
\begin{eqnarray}
&& \int\big( \frac{c_{\nu}}{2}(\zeta-\Theta)_{+}^{2}+\psi^{2}(\zeta-\Theta)_{+} \big)dx
+\int_{0}^{t}\int_{\Omega_{2}}\big(\frac{\mu\psi_{x}^{2}}{v}(\zeta-\Theta)_{+}
+\frac{\kappa\zeta_{x}^{2}}{v}\big)dxds\notag\\
&&=\int \big( \frac{c_{\nu}}{2}(\zeta_{0}(x)-\Theta(x,0))_{+}^{2}
+\psi_{0}^{2}(x)(\zeta_{0}(x)-\Theta(x,0))_{+} \big)dx\notag\\
&& \ +\int_{0}^{t}\int \frac{R\zeta-2p_{+}\phi-R\Theta}{v}\psi_{x}(\zeta-\Theta)_{+}dxds
-\int_{0}^{t}\int\frac{R\zeta-p_{+}\phi}{v}U_{x}(\zeta-\Theta)_{+}dxds\notag\\
&& \ +\kappa\int_{0}^{t}\int_{\Omega_{2}}\frac{\zeta_{x}\Theta_{x}}{V}dxds
-\kappa\int_{0}^{t}\int_{\Omega_{2}}\frac{\phi\Theta_{x}^{2}}{vV}dxds
+2\mu\int_{0}^{t}\int\frac{\psi_{x}U_{x}}{V}(\zeta-\Theta)_{+}dxds\notag\\
&& \ -\mu\int_{0}^{t}\int\frac{\phi U_{x}^{2}}{vV}(\zeta-\Theta)_{+}dxds+2\int_{0}^{t}\int_{\Omega_{2}}\frac{R\zeta-p_{+}\phi}{v}\psi\zeta_{x}dxds\label{91}\\
&& \ -2\int_{0}^{t}\int_{\Omega_{2}}\frac{R\zeta-p_{+}\phi}{v}\psi\Theta_{x}dxds
-2\mu\int_{0}^{t}\int_{\Omega_{2}}\frac{\psi\psi_{x}\zeta_{x}}{v}dxds
+2\mu \int_{0}^{t}\int_{\Omega_{2}}\frac{\phi\psi U_{x}}{vV}\zeta_{x}dxds\notag\\
&& \ +2\mu \int_{0}^{t}\int_{\Omega_{2}}\frac{\psi\psi_{x}\Theta_{x}}{v}dxds
-2\mu \int_{0}^{t}\int_{\Omega_{2}}\frac{\phi\psi}{vV}U_{x}\Theta_{x}dxds
-2\int_{0}^{t}\int\psi R_{1}(\zeta-\Theta)_{+}dxds\notag\\
&& \ -\int_{0}^{t}\int R_{2}(\zeta-\Theta)_{+}dxds
-c_{\nu}\int_{0}^{t}\int\partial_{t}\Theta(\zeta-\Theta)_{+}dxds
-\int_{0}^{t}\int_{\Omega_{2}}\psi^{2}\partial_{t}\Theta dxds\notag\\
&& \ -\frac{1}{c_{\nu}}\int_{0}^{t}\int_{\Omega_{2}}\psi^{2}
\big( \frac{R\zeta+R\Theta}{v}\psi_{x}+\frac{R\zeta-p_{+}\phi}{v}U_{x} \big)dxds
-\frac{1}{c_{\nu}}\int_{0}^{t}\int_{\Omega_{2}}\psi^{2}R_{2}dxds\notag\\
&& \ +\frac{\mu}{c_{\nu}}\int_{0}^{t}\int_{\Omega_{2}}\psi^{2}
\big( \frac{\psi_{x}^{2}+2\psi_{x}U_{x}}{v}-\frac{\phi U_{x}^{2}}{vV} \big)dxds
+\frac{\kappa}{c_{\nu}}\int_{0}^{t}\int_{\Omega_{2}}\psi^{2}
\big( \frac{\theta_{x}}{v}-\frac{\Theta_{x}}{V} \big)_{x}dxds\notag\\
&&=\int \big( \frac{c_{\nu}}{2}(\zeta_{0}(x)-\Theta_{0}(x))_{+}^{2}
+\psi_{0}^{2}(x)(\zeta_{0}(x)-\Theta_{0}(x))_{+} \big)dx+\sum_{i=1}^{20}I_{i}.\notag
\end{eqnarray}
Now the corresponding estimates are given below for $I_{i}(1\leq i\leq 20)$. First, by \eqref{new} and \eqref{v}, one derives that
\begin{equation}
\begin{split}
|I_{1}|&=\Big|\int_{0}^{t}\int \frac {R \zeta - 2p_{+} \phi -R \Theta} {v}
\psi_{x} (\zeta-\Theta)_{+} dx ds \Big| \\
&\leq \frac{\mu}{4}\int_{0}^{t}\int\frac{\psi_{x}^{2}}{v}(\zeta-\Theta)_{+}dxds
+C\int_{0}^{t}\int(\zeta^{2}+\phi^{2}+1)(\zeta-\Theta)_{+}dxds\\
&\leq \frac{\mu}{4}\int_{0}^{t}\int\frac{\psi_{x}^{2}}{v}(\zeta-\Theta)_{+}dxds
+C\int_{0}^{t}\int(\zeta^{2}+1)(\zeta-\Theta)_{+}dxds\\
&\leq \frac{\mu}{4}\int_{0}^{t}\int\frac{\psi_{x}^{2}}{v}(\zeta-\Theta)_{+}dxds
+C\int_{0}^{t}\int(\zeta+1)\Big( \zeta-\frac{1}{2}\Theta \Big)_{+}^{2}dxds\\
&\leq \frac{\mu}{4}\int_{0}^{t}\int\frac{\psi_{x}^{2}}{v}(\zeta-\Theta)_{+}dxds
+C\int_{0}^{t}\max_{x\in\mathbb{R}}\Big( \zeta-\frac{1}{2}\Theta \Big)_{+}^{2}
\int_{\{ \zeta>\Theta/2 \}}(\zeta+1)dxds\\
&\leq \frac{\mu}{4}\int_{0}^{t}\int\frac{\psi_{x}^{2}}{v}(\zeta-\Theta)_{+}dxds
+C\int_{0}^{t}\max_{x\in\mathbb{R}}\Big( \zeta-\frac{1}{2}\Theta \Big)_{+}^{2}ds.\label{37}
\end{split}
\end{equation}
Next, duo to \eqref{new} and Cauchy's inequality, it yields that
\begin{equation}
\begin{split}
&|I_{2}|+|I_{3}|+|I_{4}|\\
&\leq C\int_{0}^{t}\int \big( \frac{\zeta_{x}^{2}}{v\theta^{2}}+\frac{\phi^{2}
\Theta_{x}^{2}}{v\theta^{2}} \big)dxds
+C\int_{0}^{t}\int_{\Omega_{2}}\theta^{2}\Theta_{x}^{2}dxds+C\int_{0}^{t}\int_{\Omega_{2}}(\zeta+1)(\zeta-\Theta)_{+}dxds\\
&\leq C\int_{0}^{t}\max_{x\in\mathbb{R}}(\zeta-\frac{1}{2}\Theta)_{+}^{2}ds+C.\notag
\end{split}
\end{equation}
Similarly, one gets
\begin{equation}
\begin{split}
&|I_{5}|+|I_{6}|+|I_{7}|\\
&\leq \frac{\mu}{4}\int_{0}^{t}\int\frac{\psi_{x}^{2}}{v}(\zeta-\Theta)_{+}dxds
+\frac{\kappa}{8}\int_{0}^{t}\int_{\Omega_{2}}\frac{\zeta_{x}^{2}}{v}dxds\\
&\quad  +C\int_{0}^{t}\int U_{x}^{2}(\zeta-\Theta)_{+}dxds+C\int_{0}^{t}\int_{\Omega_{2}}(\zeta+1)^{2}\psi^{2}dxds\\
&\leq \frac{\mu}{4}\int_{0}^{t}\int\frac{\psi_{x}^{2}}{v}(\zeta-\Theta)_{+}dxds
+\frac{\kappa}{8}\int_{0}^{t}\int_{\Omega_{2}}\frac{\zeta_{x}^{2}}{v}dxds\\
&\qquad+C\int_{0}^{t}\max_{x\in\mathbb{R}}(\zeta-\frac{\Theta}{2})_{+}^{2}ds
+C\int_{0}^{t}\int U_{x}^{4}dxds\\
&\leq \frac{\mu}{4}\int_{0}^{t}\int\frac{\psi_{x}^{2}}{v}(\zeta-\Theta)_{+}dxds
+\frac{\kappa}{8}\int_{0}^{t}\int_{\Omega_{2}}\frac{\zeta_{x}^{2}}{v}dxds+C\int_{0}^{t}\max_{x\in\mathbb{R}}(\zeta-\frac{\Theta}{2})_{+}^{2}ds+C.\notag
\end{split}
\end{equation}
With \eqref{25}, \eqref{v}, \eqref{41}, and Cauchy's inequality, it holds that
\begin{equation}
\begin{split}
&|I_{8}|+|I_{9}|+|I_{11}|\\
&\leq C\int_{0}^{t}\int_{\Omega_{2}}(\zeta+1)|\psi||\Theta_{x}|dxds
+C\int_{0}^{t}\int_{\Omega_{2}}|\psi||\psi_{x}||\zeta_{x}|dxds\\
&\quad +C\int_{0}^{t}\int_{\Omega_{2}}\frac{\psi_{x}^{2}}{v\theta}dxds
+C\int_{0}^{t}\int_{\Omega_{2}}\theta\psi^{2}\Theta_{x}^{2}dxds\\
&\leq \frac{\kappa}{8}\int_{0}^{t}\int_{\Omega_{2}}\frac{\zeta_{x}^{2}}{v}dxds
+C\int_{0}^{t}\max_{x\in\mathbb{R}}\big(\zeta-\frac{\Theta}{2}\big)_{+}^{2}
\int_{\Omega_{2}}(\psi^{2}+\Theta_{x}^{2})dxds +C\int_{0}^{t}\int\psi^{2}\psi_{x}^{2}dxds+C\\
&\leq \frac{\kappa}{8}\int_{0}^{t}\int_{\Omega_{2}}\frac{\zeta_{x}^{2}}{v}dxds
+C\int_{0}^{t}\max_{x\in\mathbb{R}}\big(\zeta-\frac{\Theta}{2}\big)_{+}^{2}ds
+C\int_{0}^{t}\int\psi^{2}\psi_{x}^{2}dxds+C.\notag
\end{split}
\end{equation}
Obviously, a direct calculation gives
\begin{equation}
\begin{split}
&|I_{10}|+|I_{12}|+|I_{13}|+|I_{14}|+|I_{15}|\\
&\leq \frac{\kappa}{8}\int_{0}^{t}\int_{\Omega_{2}}\frac{\zeta_{x}^{2}}{v}dxds
+C\int_{0}^{t}\int\psi^{2}U_{x}^{2}dxds
+C\int_{0}^{t}\int_{\Omega_{2}}(\psi^{2}+1)|U_{x}||\Theta_{x}|dxds\\
&\quad +C\int_{0}^{t}\max_{x\in\mathbb{R}}(\zeta-\frac{\Theta}{2})_{+}^{2}ds
+C\int_{0}^{t}\int_{\Omega_{2}}\left( \psi^{2}R_{1}^{2}+R_{2}^{2}+U_{x}^{2} \right)dxds\\
&\leq \frac{\kappa}{8}\int_{0}^{t}\int_{\Omega_{2}}\frac{\zeta_{x}^{2}}{v}dxds
+C\int_{0}^{t}\max_{x\in\mathbb{R}}(\zeta-\frac{\Theta}{2})_{+}^{2}ds
+C\int_{0}^{t}\frac{1}{(1+s)^{3/2}}\big(\int\psi^{2}dx+1\big)ds\\
&\leq \frac{\kappa}{8}\int_{0}^{t}\int_{\Omega_{2}}\frac{\zeta_{x}^{2}}{v}dxds
+C\int_{0}^{t}\max_{x\in\mathbb{R}}(\zeta-\frac{\Theta}{2})_{+}^{2}ds+C,\notag
\end{split}
\end{equation}
and
\begin{equation}
\begin{split}
&|I_{16}|+|I_{17}|+|I_{18}|+|I_{19}|\\
&\leq C\int_{0}^{t}\int\psi^{2}\psi_{x}^{2}dxds+C\int_{0}^{t}\int\psi^{2}|U_{x}|dxds
+C\int_{0}^{t}\frac{1}{(1+s)^{2}}\int\psi^{2}dxds\\
&\quad +C\int_{0}^{t}\int_{\Omega_{2}}\psi^{2}(\zeta^{2}+1)dxds\\
&\leq C\int_{0}^{t}\int\psi^{2}\psi_{x}^{2}dxds+C\int_{0}^{t}\max_{x\in\mathbb{R}}\psi^{4}ds
+C\int_{0}^{t}\max_{x\in\mathbb{R}}(\zeta-\frac{\Theta}{2})_{+}^{2}ds+C.\notag
\end{split}
\end{equation}
Finally,  In order to get the estimate  of the last term of \eqref{91}, the following cut-off function is constructed
\begin{equation}\label{eta}
\displaystyle\varphi_{\eta}(z)=\left\{ \begin{array}{l}
1,\quad z>\eta,\\ 
\displaystyle \frac{z}\eta, \quad 0<z\leq\eta,\\
\displaystyle 0,\quad z\leq 0,
       \end{array}\right.
\end{equation}
pluging $\varphi_{\eta}(z)$ into $I_{20}$ in limit form, 
\begin{eqnarray}
I_{20}&=&\frac{\kappa}{c_{\nu}}\int_{0}^{t}\int_{\Omega_{2}}\psi^{2}
\left( \frac{\theta_{x}}{v}-\frac{\Theta_{x}}{V} \right)_{x}dxds\notag\\
&=&\frac{\kappa}{c_{\nu}}\lim_{\eta\rightarrow 0}\int_{0}^{t}\int\varphi_{\eta}(\zeta-\Theta)\psi^{2}\big( \frac{\zeta_{x}
-\Theta_{x}}{v} \big)_{x}dxds+\frac{\kappa}{c_{\nu}}\int_{0}^{t}\int_{\Omega_{2}}\psi^{2}\big( \frac{2\Theta_{x}}{v}
-\frac{\Theta_{x}}{V} \big)_{x}dxds\notag\\
&\overset{\text{def}}{=}& I_{20}^{1}+I_{20}^{2},\label{I20}
\end{eqnarray}
where, for the first term on the right hand side of \eqref{I20}, Lebesgue's dominated convergence theorem shows that
\begin{equation}
\begin{split}
I_{20}^{1}&=-\frac{2\kappa}{c_{\nu}}\lim_{\eta\rightarrow 0}\int_{0}^{t}\int\varphi_{\eta}(\zeta-\Theta)\psi\psi_{x}\frac{\zeta_{x}-\Theta_{x}}{v}dxds\\
&\qquad\qquad-\frac{\kappa}{c_{\nu}}\lim_{\eta\rightarrow 0}\int_{0}^{t}\int\varphi_{\eta}^{'}(\zeta-\Theta)\frac{\psi^{2}(\zeta_{x}-\Theta_{x})^{2}}{v}dxds\\
&\leq C\int_{0}^{t}\int_{\Omega_{2}}\frac{|\psi\psi_{x}\zeta_{x}|+|\psi\psi_{x}\Theta_{x}|}{v}dxds\\
&\leq \frac{\kappa}{8}\int_{0}^{t}\int_{\Omega_{2}}\frac{\zeta_{x}^{2}}{v}dxds
+C\int_{0}^{t}\int\psi^{2}\psi_{x}^{2}dxds+C\int_{0}^{t}\int\frac{\psi_{x}^{2}}{v\theta}dxds +C\int_{0}^{t}\int_{\Omega_{2}}\theta\psi^{2}\Theta_{x}^{2}dxds\\
&\leq \frac{\kappa}{8}\int_{0}^{t}\int_{\Omega_{2}}\frac{\zeta_{x}^{2}}{v}dxds
+C\int_{0}^{t}\int\psi^{2}\psi_{x}^{2}dxds
+C\int_{0}^{t}\max_{x\in\mathbb{R}}(\zeta-\frac{\Theta}{2})_{+}^{2}ds+C,\notag
\end{split}
\end{equation}
and in the second inequality we have used $\varphi_{\eta}\in[0,1]$ and $0\leq\varphi_{\eta}^{'}$. Similarly, for the $I_{20}^2$, one has
\begin{equation}
\begin{split}
I_{20}^{2}&=\frac{\kappa}{c_{\nu}}\int_{0}^{t}\int_{\Omega_{2}}\psi^{2}\big( \frac{2\Theta_{xx}}{v}
-\frac{\Theta_{xx}}{V}-\frac{2\Theta_{x}v_{x}}{v^{2}}+\frac{\Theta_{x}V_{x}}{V^{2}} \big)dxds\\
&\leq C\int_{0}^{t}\int_{\Omega_{2}}\psi^{2}\left( |\Theta_{xx}|+\Theta_{x}^{2} +|\phi_{x}||\Theta_{x}|\right)dxds\\
&\leq C\int_{0}^{t}\max_{x\in\mathbb{R}}\psi^{4}ds
+C\int_{0}^{t}\int\frac{\theta\phi_{x}^{2}}{v^{3}}dxds
+C.\label{36}
\end{split}
\end{equation}
Furthermore, noting that by \eqref{new} and Lemma 3.2, one derives
\begin{equation}
\begin{split}
&\int_{0}^{t}\int\left( \theta\psi_{x}^{2}+\zeta_{x}^{2} \right)dxds\\
&\leq \int_{0}^{t}\int_{\{\zeta>2\Theta\}}(\theta\psi_{x}^{2}+\zeta_{x}^{2})dxds+\int_{0}^{t}\int_{\{\zeta\leq 2\Theta\}}(\theta\psi_{x}^{2}+\zeta_{x}^{2})dxds\\
&\leq \int_{0}^{t}\int_{\{\zeta>2\Theta\}}\big( \frac{3}{2}\psi_{x}^{2}\zeta
+\frac{\zeta_{x}^{2}}{v}v \big)dxds
+\int_{0}^{t}\int_{\{\zeta\leq 2\Theta\}}\big( \frac{\psi_{x}^{2}}{\theta}
+\frac{\zeta_{x}^{2}}{\theta^{2}} \big)\theta^{2}dxds\\
&\leq C\int_{0}^{t}\int_{\{\zeta>2\Theta\}}\big( \psi_{x}^{2}\zeta+\frac{\zeta_{x}^{2}}{v} \big)dxds
+C\int_{0}^{t}\int_{\{\zeta\leq 2\Theta\}}\big( \frac{\psi_{x}^{2}}{\theta}+\frac{\zeta_{x}^{2}}{\theta^{2}} \big)dxds\\
&\leq C\int_{0}^{t}\int_{\Omega_{2}}\big( \frac{\psi_{x}^{2}}{v}(\zeta-\Theta)_{+}+\frac{\zeta_{x}^{2}}{v} \big)dxds+C.\label{38}
\end{split}
\end{equation}
Substituting the estimates \eqref{37}-\eqref{36} into \eqref{91}, and by using \eqref{38}, it yields that
\begin{equation}
\begin{split}
&\int(\zeta-\Theta)_{+}^{2}dx+\int_{0}^{t}\int\left( \theta\psi_{x}^{2}+\zeta_{x}^{2} \right)dxds\\
&\leq C\int_{0}^{t}\Big( \max_{x\in\mathbb{R}}\big(\zeta-\frac{\theta}{2}\big)_{+}^{2}+\max_{x\in\mathbb{R}}\psi^{4} \Big)ds\\
&\qquad\qquad+C\int_{0}^{t}\int\psi^{2}\psi_{x}^{2}dxds +C\int_{0}^{t}\int\frac{\theta\phi_{x}^{2}}{v^{3}}dxds+C.\label{aa}
\end{split}
\end{equation}
Now, rewriting \eqref{2} as following 
\begin{gather}
\left( \mu\frac{\bar{v}_{x}}{\bar{v}}-\psi \right)_{t}-p_{x}=R_{1},\label{40}
\end{gather}
where $\bar{v}=\frac vV$, multiplying \eqref{40} by $\frac{\bar{v}_{x}}{\bar{v}}$, one gets
\begin{equation}
\begin{split}
&\Big( \frac{\mu}{2}\big(\frac{\bar{v}_{x}}{\bar{v}}\big)^{2}-\psi\frac{\bar{v}_{x}}{\bar{v}} \Big)_{t}
+\big( \psi\frac{\bar{v}_{t}}{\bar{v}} \big)_{x}+\frac{ R\theta}{v}\big( \frac{\bar{v}_{x}}{\bar{v}} \big)^{2}\\
&=\frac{ R\zeta_{x}}{v}\frac{\bar{v}_{x}}{\bar{v}}-\frac{ R\theta\Theta_{x}}{v}\big( \frac{1}{\Theta}-\frac{1}{\theta} \big)\frac{\bar{v}_{x}}{\bar{v}}
+\frac{\psi_{x}^{2}}{v}+\psi_{x}U_{x}\big( \frac{1}{v}-\frac{1}{V} \big)+ R_{1}\frac{\bar{v}_{x}}{\bar{v}}\\
&\leq \frac{ R\theta}{2v}\big( \frac{\bar{v}_{x}}{\bar{v}} \big)^{2}+ C\big( \frac{\zeta_{x}^{2}}{\theta}+\psi_{x}^{2}+\frac{\zeta^{2}\Theta_{x}^{2}}{v\theta}+
U_{x}^{2}\phi^{2} \big)
 + R_{1}\frac{\bar{v}_{x}}{\bar{v}}.\label{92}\\
\end{split}
\end{equation}
Considering the following inequality which can be obtained by Cauchy's inequality and a direct calculation
\begin{gather}
\frac{\phi_{x}^{2}}{2v^{2}}-C\phi^{2}\Theta_{x}^{2}\leq \big( \frac{\bar{v}_{x}}{\bar{v}} \big)^{2}\leq
\frac{\phi_{x}^{2}}{v^{2}}+C\phi^{2}\Theta_{x}^{2},\label{93}
\end{gather}
integrating \eqref{92} over $\mathbb{R}\times(0,t)$, one has
\begin{eqnarray}
&&\int\phi_{x}^{2}dx+\int_{0}^{t}\int\frac{\theta\phi_{x}^{2}}{v^{3}}dxds\notag\\
&&\leq C\int_{0}^{t}\int\big( \frac{\zeta_{x}^{2}}{\theta}+\psi_{x}^{2} \big)dxds+
C\int_{0}^{t}\int\big(\frac{\zeta^{2}\Theta_{x}^{2}}{v\theta}+\theta\phi^{2}\Theta_{x}^{2}\big)dxds\notag\\
&&\quad +C\int_{0}^{t}\big(\int R_{1}^{2}dx\big)^{\frac{1}{2}}\big(\int(\phi_{x}^{2}+\phi^{2}\Theta_{x}^{2})dx \big)^{\frac{1}{2}}ds+C\notag\\
&&\leq C\int_{0}^{t}\int\big( \frac{\zeta_{x}^{2}}{\theta}+\psi_{x}^{2} \big)dxds+
C\int_{0}^{t}\int\big(\frac{\zeta^{2}\Theta_{x}^{2}}{v\theta}+\theta\phi^{2}\Theta_{x}^{2}\big)dxds\notag\\
&&\quad +C\int_{0}^{t}(1+s)^{-\frac{5}{4}}\Big(\int\big( \phi_{x}^{2}+\phi^{2}\Theta_{x}^{2}\big )dx \Big)ds+C \label{101}\\
&&\leq C\int_{0}^{t}\int\Big( \frac{\zeta_{x}^{2}}{\theta^{2}}+\frac{\psi_{x}^{2}}{\theta} \Big)dxds
+\varepsilon\int_{0}^{t}\int(\zeta_{x}^{2}+\theta\psi_{x}^{2})dxds +C\int_{0}^{t}\int_{\{\zeta>\Theta\}}\theta\Theta_{x}^{2}dxds\notag\\
&&\quad+
C\int_{0}^{t}\int_{\{ \zeta\leq\Theta \}}\frac{\phi^{2}\Theta_{x}^{2}}{v\theta^{2}}dxds +C\int_{0}^{t}\int_{\{ \zeta>\Theta \}}\frac{\zeta^{2}\Theta_{x}^{2}}{\theta}dxds
+C\int_{0}^{t}\int_{\{ \zeta\leq\Theta \}}\frac{\zeta^{2}\Theta_{x}^{2}}{\theta^{2}}dxds\notag\\
&&\quad +C\int_{0}^{t}(1+s)^{-\frac{5}{4}}\int\phi_{x}^{2}dxds+C\notag\\
&&\leq C\int_{0}^{t}\max_{x\in\mathbb{R}}\big( \zeta-\frac{\Theta}{2} \big)_{+}^{2}ds+\varepsilon\int_{0}^{t}\int(\zeta_{x}^{2}+\theta\psi_{x}^{2})dxds +C\int_{0}^{t}(1+s)^{-\frac{5}{4}}\int\phi_{x}^{2}dxds+C,\notag
\end{eqnarray}
which along with \eqref{aa} and  Gronwall's inequality, one obtains
\begin{equation}
\begin{split}
&\int\big(\phi_{x}^{2}+(\zeta-\Theta)_{+}^{2}\big)dx
+\int_{0}^{t}\int\big( \theta(\psi_{x}^{2}+\phi_{x}^{2})+\zeta_{x}^{2} \big)dxds\\
&\leq C\int_{0}^{t}\big( \max_{x\in\mathbb{R}}(\zeta-\frac{\Theta}{2})_{+}^{2}
+\max_{x\in\mathbb{R}}\psi^{4} \big)ds
+C\int_{0}^{t}\int\psi^{2}\psi_{x}^{2}dxds+C.\label{45}
\end{split}
\end{equation}
In order to get the estimate of the second  integral term on the right hand side of \eqref{45},  multiplying \eqref{2} by $\psi^{3}$, and integrating the resultant over $\mathbb{R}\times[0,t]$, one derives
\begin{equation}
\begin{split}
&\frac{1}{4}\int\psi^{4}dx+3\mu\int_{0}^{t}\int\frac{\psi^{2}\psi_{x}^{2}}{v}dxds\\
&=\frac{1}{4}\int\psi_{0}^{4}dx+3R\int_{0}^{t}\int\frac{\zeta\psi^{2}\psi_{x}}{v}dxds
-3p_{+}\int_{0}^{t}\int\frac{\phi\psi^{2}\psi_{x}}{v}dxds\\
&\quad +3\mu\int_{0}^{t}\int\frac{\phi U_{x}}{vV}\psi^{2}\psi_{x}dxds-\int_{0}^{t}\int R_{1}\psi^{3}dxds\\
&\overset{\text{def}}{=} \frac{1}{4}\int\psi_{0}^{4}dx+\sum_{i=1}^{4}J_{i}.\label{44}
\end{split}
\end{equation}
Due to \eqref{new} and \eqref{v}, one has
\begin{eqnarray}
|J_{1}|&=&3R\int_{0}^{t}\int_{\{ \zeta>\Theta \}}\frac{\zeta\psi^{2}\psi_{x}}{v}dxds
+3R\int_{0}^{t}\int_{\{ \zeta\leq\Theta \}}\frac{\zeta\psi^{2}\psi_{x}}{v}dxds\notag\\
&\leq& \mu\int_{0}^{t}\int_{\{ \zeta>\Theta \}}\frac{\psi^{2}\psi_{x}^{2}}{v}dxds
+C\int_{0}^{t}\int_{\{ \zeta>\Theta \}}\zeta^{2}\psi^{2}dxds\notag\\
&&\quad+\int_{0}^{t}\int_{\{ \zeta\leq\Theta \}}\psi_{x}^{2}dxds
+C\int_{0}^{t}\int_{\{ \zeta\leq\Theta \}}\zeta^{2}\psi^{4}dxds\label{J1}\\
&\leq& \mu\int_{0}^{t}\int\frac{\psi^{2}\psi_{x}^{2}}{v}dxds
+C\int_{0}^{t}\max_{x\in\mathbb{R}}\big( \zeta-\frac{\Theta}{2} \big)_{+}^{2}ds +C\int_{0}^{t}\int\frac{\psi_{x}^{2}}{\theta}dxds\notag\\
&&\qquad\qquad\qquad
+C\int_{0}^{t}\max_{x\in\mathbb{R}}\psi^{4}ds\notag\\
&\leq& \mu\int_{0}^{t}\int\frac{\psi^{2}\psi_{x}^{2}}{v}dxds
+C\int_{0}^{t}\Big( \max_{x\in\mathbb{R}}\big( \zeta-\frac{\Theta}{2} \big)_{+}^{2}
+\max_{x\in\mathbb{R}}\psi^{4} \Big)ds+C.\notag
\end{eqnarray}
By using \eqref{new} and \eqref{v} again, it leads to 
\begin{equation}
\begin{split}
|J_{2}|&\leq \varepsilon\int_{0}^{t}\int\psi_{x}^{2}dxds+C\int_{0}^{t}\int\phi^{2}\psi^{4}dxds\\
&\leq \varepsilon\int_{0}^{t}\int\theta\psi_{x}^{2}dxds+C\int_{0}^{t}\max_{x\in\mathbb{R}}\psi^{4}ds+C.
\end{split}
\end{equation}
Combining \eqref{new} and  \eqref{v} with \eqref{41}, one gets
\begin{equation}
\begin{split}
|J_{3}|&\leq \mu\int_{0}^{t}\int\frac{\psi^{2}\psi_{x}^{2}}{v}dxds
+C\int_{0}^{t}\int\phi^{2}\psi^{2}U_{x}^{2}dxds\\
&\leq \mu\int_{0}^{t}\int\frac{\psi^{2}\psi_{x}^{2}}{v}dxds+C\int_{0}^{t}\int U_{x}^{2}dxds
+C\int_{0}^{t}\int\psi^{4}U_{x}^{2}dxds\\
&\leq \mu\int_{0}^{t}\int\frac{\psi^{2}\psi_{x}^{2}}{v}dxds+C\int_{0}^{t}\max_{x\in\mathbb{R}}\psi^{4}ds
+C.
\end{split}
\end{equation}
Similarly,from \eqref{new}, one obtains
\begin{equation}
\begin{split}
|J_{4}|&\leq\int_{0}^{t}\int\psi^{6}dxds+\int_{0}^{t}\int R_{1}^{2}dxds\leq C\int_{0}^{t}\max_{x\in\mathbb{R}}\psi^{4}ds+C.\label{J4}
\end{split}
\end{equation}
Putting the estimates \eqref{J1}-\eqref{J4} into \eqref{44},  one has
\begin{equation}\label{phi^4}
\begin{split}
&\frac{1}{4}\int\psi^{4}dx+\mu\int_{0}^{t}\int\frac{\psi^{2}\psi_{x}^{2}}{v}dxds\\
&\leq C\int_{0}^{t}\Big( \max_{x\in\mathbb{R}}\big( \zeta-\frac{\Theta}{2} \big)_{+}^{2}+\max_{x\in\mathbb{R}}\psi^{4} \Big)ds
+\varepsilon\int_{0}^{t}\int\theta\psi_{x}^{2}dxds+C,
\end{split}
\end{equation}
together \eqref{phi^4} with \eqref{45} and choosing $\varepsilon$ suitably small, one derives
\begin{equation}
\begin{split}
&\int\left(\phi_{x}^{2}+(\zeta-\Theta)_{+}^{2}+\psi^{4}\right)dx
+\int_{0}^{t}\int\Big(\big(\psi_{x}^{2}+\phi_{x}^{2}\big)\theta
+\psi^{2}\psi_{x}^{2}+\zeta_{x}^{2}\Big)dxds\\
&\leq C\int_{0}^{t}\Big( \max_{x\in\mathbb{R}}\big( \zeta-\frac{\Theta}{2} \big)_{+}^{2}+\max_{x\in\mathbb{R}}\psi^{4} \Big)ds+C,\label{751}
\end{split}
\end{equation}
at this point, all that remains is to give the estimate of the first terms on the right hand side of \eqref{751}. In fact, by using standard calculations, for  sufficiently small $\varepsilon>0$, it yields that
\begin{equation}
\begin{split}
&\big( \zeta-\frac{\Theta}{2} \big)_{+}^{2}=2\int_{-\infty}^{x}\big( \zeta-\frac{\Theta}{2} \big)_{+}
\big( \zeta_{x}-\frac{\Theta_{x}}{2} \big)dx\\
&\leq \int\big( \zeta-\frac{\Theta}{2} \big)_{+}\big( |\zeta_{x}|+|\Theta_{x}| \big)dx\\
&\leq\varepsilon\int\big( \zeta-\frac{\Theta}{2} \big)_{+}^{2}\theta dx+\frac{C}{\varepsilon}
\int_{\{ \zeta>\Theta/2 \}}
\big( \frac{\zeta_{x}^{2}}{\theta}+\frac{\Theta_{x}^{2}}{\theta} \big)dx\\
&\leq 3\varepsilon\int\big( \zeta-\frac{\Theta}{2} \big)_{+}^{2}\zeta dx+\frac{C}{\varepsilon}\int\frac{\zeta_{x}^{2}}{\theta}dx+\frac{C}{\varepsilon}\int_{\{\Theta/2<\zeta\leq\Theta\}}\frac{\zeta^{2}\Theta_{x}^{2}}{\theta^{2}}dx
+\frac{C}{\varepsilon}\int_{\{ \zeta>\Theta \}}\zeta\Theta_{x}^{2}dx\\
&\leq 3\varepsilon\max_{x\in\mathbb{R}}\big( \zeta-\frac{\Theta}{2} \big)_{+}^{2}
\int_{\{ \zeta>\Theta/2 \}}\zeta dx
+\frac{C}{\varepsilon}\int\frac{\zeta_{x}^{2}}{\theta}dx\\
&\quad  \quad+\frac{C}{\varepsilon}\int\frac{\zeta^{2}\Theta_{x}^{2}}{\theta^{2}}dx
+\frac{C}{\varepsilon}\max_{x\in\mathbb{R}}\big( \zeta-\frac{\Theta}{2} \big)_{+}\int_{\{\zeta>\Theta\}}\Theta_{x}^{2}dx\\
&\leq \varepsilon C\max_{x\in\mathbb{R}}\big( \zeta-\frac{\Theta}{2} \big)_{+}^{2}+\frac{C}{\varepsilon}\int\frac{\zeta_{x}^{2}}{\theta}dx +\frac{C}{\varepsilon}\int\frac{\zeta^{2}\Theta_{x}^{2}}{\theta^{2}}dx
+\frac{C}{\varepsilon^{3}}\big(\int_{\{\zeta>\Theta\}}\Theta_{x}^{2}dx\big)^{2},\notag
\end{split}
\end{equation}
which leads to
\begin{gather}
\max_{x\in\mathbb{R}}\big( \zeta-\frac{\Theta}{2} \big)_{+}^{2}\leq C\int\frac{\zeta_{x}^{2}}{\theta}dx
+C\int\frac{\zeta^{2}\Theta_{x}^{2}}{\theta^{2}}dx+C\int\Theta_{x}^{4}dx.\label{46}
\end{gather}
For last term, a direct calculation gives
\begin{equation}
\begin{split}
\psi^{4}&=4\int_{-\infty}^{x}\psi^{3}\psi_{x}dx\leq 4\int_{\{ \zeta>\Theta \}}|\psi|^{3}|\psi_{x}|dx
+4\int_{\{ \zeta\leq\Theta \}}|\psi|^{3}|\psi^{x}|dx\\
&\leq \varepsilon\int_{\{\zeta>\Theta\}}|\psi|^{5}\theta^{\frac{1}{2}}dx+\frac{C}{\varepsilon}
\int_{\{ \zeta>\Theta \}}\frac{\psi_{x}^{2}|\psi|}{\theta^{\frac{1}{2}}}dx +\varepsilon\int_{\{\zeta\leq\Theta\}}\psi^{6}\theta dx+\frac{C}{\varepsilon}\int_{\{\zeta\leq\Theta\}}\frac{\psi_{x}^{2}}{\theta}dx\\
&\leq \varepsilon\max_{x\in\mathbb{R}}\psi^{4}\int_{\{ \zeta>\Theta \}}(\psi^{2}+\theta)dx
+C\varepsilon\max_{x\in\mathbb{R}}\psi^{4}\int_{\{ \zeta\leq\Theta \}}\psi^{2}dx+\frac{C}{\varepsilon}\int\psi_{x}^{2}\big( \frac{|\psi|}{\theta^{\frac{1}{2}}}
+\frac{1}{\theta} \big)dx\\
& \leq \varepsilon C\max_{x\in\mathbb{R}}\psi^{4}+\frac{C}{\varepsilon}\int\psi_{x}^{2}
\big( \frac{|\psi|}{\theta^{\frac{1}{2}}}+\frac{1}{\theta} \big)dx.\label{47}
\end{split}
\end{equation}
Substituting \eqref{46} and \eqref{47} into \eqref{45}, and choosing $\varepsilon$ suitably small, one obtains
\begin{equation}
\begin{split}
&\int\left(\phi_{x}^{2}+(\zeta-\Theta)_{+}^{2}+\psi^{4}\right)dx
+\int_{0}^{t}\int\Big( \big( \theta+\psi^{2} \big)\psi_{x}^{2}+\theta\phi_{x}^{2}+\zeta_{x}^{2} \Big)dxds\\
&\leq C\int_{0}^{t}\int\left( \frac{\zeta_{x}^{2}}{\theta}
+\Big( \frac{|\psi|}{\theta^{\frac{1}{2}}}+\frac{1}{\theta} \Big)\psi_{x}^{2} \right)dxds+C\\
&\leq \frac{1}{2}\int_{0}^{t}\int\big(\zeta_{x}^{2}+\psi^{2}\psi_{x}^{2}\big)dxds
+C\int_{0}^{t}\int\left( \frac{\zeta_{x}^{2}}{\theta^{2}}+\frac{\psi_{x}^{2}}{\theta} \right)dxds+C\\
&\leq \frac{1}{2}\int_{0}^{t}\int\big(\zeta_{x}^{2}+\psi^{2}\psi_{x}^{2}\big)dxds+C.\label{49}
\end{split}
\end{equation}
From \eqref{new} and \eqref{o}, it holds that
\begin{gather}
\int_{\{ \zeta\leq 2\Theta \}}\zeta^{2}dx\leq C\int\Phi\left( \frac{\theta}{\Theta} \right)dx\leq C_{0},\notag
\end{gather}
and
\begin{gather}
\int_{\{ \zeta> 2\Theta \}}\zeta^{2}dx\leq 4\int_{\{ \zeta>2\Theta \}}\left( \zeta-\Theta \right)^{2}dx\leq 4\int
\left( \zeta-\Theta \right)_{+}^{2}dx.\label{50}
\end{gather}
Combining \eqref{49}-\eqref{50}, \eqref{61} is immediately obtained, and  the proof of Lemma 3.3 is completed.
\end{proof}

\vskip 0.3cm
Below, in order to obtain the higher order derivative estimates of the solutions, a weighted estimate is required, for which the following weighting function is introduced
\begin{equation}\label{sigma}
\tilde{\sigma}(t) =
\begin{cases}
t & \text{}  0\leq t \leq 1,\\
1 & \text{} t>1.
\end{cases} 
\end{equation}

\begin{lemma}\label{67}
There exists a positive constant $C$ depending on $\displaystyle\inf_{x\in \mathbb{R}}v_{0},\|\phi_{0}\|_{H^{1}(\mathbb{R})}$, $\|\psi_{0}\|_{L^{2}(\mathbb{R})}$, $\|\psi_{0}\|_{L^{4}(\mathbb{R})}$, $\|\zeta_{0}\|_{L^{2}(\mathbb{R})}$ and $C_{0}$ such that, for any given $T>0$,
\begin{gather}
\tilde\sigma\int\psi_{x}^{2}dx+\tilde\sigma^{2}\int\zeta_{x}^{2}dx+\int_{0}^{T}\tilde\sigma\int\psi_{xx}^{2}dxds
+\int_{0}^{T}\tilde\sigma^{2}\int\zeta_{xx}^{2}dxds\leq C.\label{68}
\end{gather}
\end{lemma}
\begin{proof}
Multiplying \eqref{2}  by $-\tilde\sigma\psi_{xx}$, integrating  the resultant over $\mathbb{R}\times (0,t)$, it  yields
\begin{equation}
\begin{split}
&\frac{\tilde\sigma}{2}\int\psi_{x}^{2}dx+\mu\int_{0}^{t}\tilde\sigma\int\frac{\psi_{xx}^{2}}{v}dxds\\
&=\frac{1}{2}\int_{0}^{t}\tilde\sigma^{'}\int\psi_{x}^{2}dxds+\int_{0}^{t}\tilde\sigma\int\big( \frac{R\zeta_{x}}{v}-\frac{R\theta\phi_{x}}{v^{2}}
-\frac{R\zeta-p_{+}\phi}{v^{2}}V_{x} \big)\psi_{xx}dxds\\
&\quad  -\mu\int_{0}^{t}\tilde\sigma\int\psi_{x}\big( \frac{1}{v} \big)_{x}\psi_{xx}dxds
+\mu\int_{0}^{t}\tilde\sigma\int\big( \frac{U_{x}}{V}-\frac{U_{x}}{v} \big)_{x}\psi_{xx}dxds\\
&\quad  +\int_{0}^{t}\tilde\sigma\int R_{1}\psi_{xx}dxds.\label{60}
\end{split}
\end{equation}
The following estimates are given for each term on the right-hand side of \eqref{60} in turn. First,  from \eqref{61}, one has
\begin{equation}
    \begin{split}
        \frac{1}{2}\int_{0}^{t}\tilde\sigma^{'}\int\psi_{x}^{2}dxds\leq C\int_{0}^{t}\int\left( \frac{\psi_{x}^{2}}{\theta}+\theta\psi_{x}^{2} \right)dxds\leq C.\label{95}
    \end{split}
\end{equation}
Next, due to  \eqref{new}, \eqref{v} and \eqref{61}, it derives that
\begin{eqnarray}
&&\Big| \int_{0}^{t}\tilde\sigma\int\big( \frac{R\zeta_{x}}{v}-\frac{R\theta\phi_{x}}{v^{2}}
-\frac{R\zeta-p_{+}\phi}{v^{2}}V_{x} \big)\psi_{xx}dxds \Big|\notag\\
&&\leq \frac{\mu}{8}\int_{0}^{t}\tilde\sigma\int\frac{\psi_{xx}^{2}}{v}dxds
+C\int_{0}^{t}\tilde\sigma\int\Big( \zeta_{x}^{2}+\theta^{2}\phi_{x}^{2}+(\phi^{2}+\zeta^{2})\Theta_{x}^{2} \Big)dxds\notag\\
&&\leq \frac{\mu}{8}\int_{0}^{t}\tilde\sigma\int\frac{\psi_{xx}^{2}}{v}dxds+C\max_{x\in\mathbb{R},t\geq0}(\tilde\sigma\theta)\int_{0}^{t}\int\theta\phi_{x}^{2}dxds+C\int_{0}^{t}\int_{\{ \zeta>\Theta \}}\phi^{2}\Theta_{x}^{2}dxds\notag\\
&&\quad +C\int_{0}^{t}
\int_{\{ \zeta\leq\Theta \}}\phi^{2}\Theta_{x}^{2}dxds +C\int_{0}^{t}\int_{\{ \zeta>\Theta \}}\zeta^{2}\Theta_{x}^{2}dxds+C\int_{0}^{t}
\int_{\{ \zeta\leq\Theta \}}\zeta^{2}\Theta_{x}^{2}dxds+C\notag\\
&&\leq \frac{\mu}{8}\int_{0}^{t}\tilde\sigma\int\frac{\psi_{xx}^{2}}{v}dxds+C\int_{0}^{t}\int_{\{ \zeta\leq\Theta \}}\frac{\phi^{2}\Theta_{x}^{2}}{\theta^{2}}dxds+C\int_{0}^{t}\int_{\{ \zeta\leq\Theta \}}\frac{\zeta^{2}\Theta_{x}^{2}}{\theta}dxds\notag\\
&&\quad  +C\int_{0}^{t}\max_{x\in\mathbb{R}}
\left( \zeta-\frac{\Theta}{2} \right)_{+}^{2}\int\Theta_{x}^{2}dxds+C\max_{x\in\mathbb{R},t\geq0}(\tilde\sigma\theta)+C\label{63}\\
&&\leq  \frac{\mu}{8}\int_{0}^{t}\tilde\sigma\int\frac{\psi_{xx}^{2}}{v}dxds
+C\int_{0}^{t}\max_{ x\in\mathbb{R} }\big( \zeta-\frac{\Theta}{2} \big)_{+}^{2}ds+C\max_{x\in\mathbb{R},t\geq0}(\tilde\sigma\theta)+C\notag\\
&&\leq  \frac{\mu}{8}\int_{0}^{t}\tilde\sigma\int\frac{\psi_{xx}^{2}}{v}dxds+C\max_{x\in\mathbb{R},t\geq0}(\tilde\sigma\theta)+C.\notag
\end{eqnarray}
By Cauchy's inequality, \eqref{new} and \eqref{61}, one gets
\begin{equation}
\begin{split}
&\Big| \mu\int_{0}^{t}\tilde\sigma\int\psi_{x}\big(\frac{1}{v}\big)\psi_{xx}dxds \Big|\leq C\int_{0}^{t}\tilde\sigma\int\Big( \big| \frac{\psi_{x}\phi_{x}\psi_{xx}}{v^{2}} \big|
+\big| \frac{\psi_{x}V_{x}\psi_{xx}}{v^{2}} \big| \Big)dxds\\
&\leq \frac{\mu}{16}\int_{0}^{t}\tilde\sigma\int\frac{\psi_{xx}^{2}}{v}dxds
+C\int_{0}^{t}\int\left( \psi_{x}^{2}\phi_{x}^{2}+\psi_{x}^{2}\Theta_{x}^{2} \right)dxds\\
&\leq \frac{\mu}{16}\int_{0}^{t}\tilde\sigma\int\frac{\psi_{xx}^{2}}{v}dxds
+C \int_{0}^{t}\|\psi_{x}\|_{L^{\infty}}^{2}\int\phi_{x}^{2}dxds+C\int_{0}^{t}\int\psi_{x}^{2}dxds\\
&\leq \frac{\mu}{8}\int_{0}^{t}\tilde\sigma\int\frac{\psi_{xx}^{2}}{v}dxds+C\int_{0}^{t}\int\psi_{x}^{2}dxds\leq \frac{\mu}{8}\int_{0}^{t}\tilde\sigma\int\frac{\psi_{xx}^{2}}{v}dxds+C,
\end{split}
\end{equation}
by the similarly way,  it yields that
\begin{equation}
\begin{split}
&\mu\int_{0}^{t}\tilde\sigma\int\big( \frac{U_{x}}{V}-\frac{U_{x}}{v} \big)_{x}\psi_{xx}dxds\\
&=\mu\int_{0}^{t}\tilde\sigma\int\big( \frac{U_{xx}}{V}-\frac{U_{xx}}{v}-\frac{U_{x}V_{x}}{V^{2}}
+\frac{v_{x}U_{x}}{v^{2}} \big)\psi_{xx}dxds\\
&=\mu\int_{0}^{t}\tilde\sigma\int\big( \frac{\phi U_{xx}}{V}-\frac{\phi(\phi+2V)}{v^{2}V^{2}}U_{x}V_{x}+\frac{\phi_{x}U_{x}}{v^{2}} \big)\psi_{xx}dxds\\
&\leq \frac{\mu}{8}\int_{0}^{t}\tilde\sigma\int\frac{\psi_{xx}^{2}}{v}dxds
+C\int_{0}^{t}\int\big(( U_{xx}^{2}+\Theta_{x}^{2}U_{x}^{2} )\phi^{2}+\phi^{4}U_{x}^{2}\Theta_{x}^{2}+\phi_{x}^{2}U_{x}^{2} \big)dxds\\
&\leq \frac{\mu}{8}\int_{0}^{t}\tilde\sigma\int\frac{\psi_{xx}^{2}}{v}dxds
+C\int_{0}^{t}\frac{1}{(1+s)^{2}}\int\phi_{x}^{2}dxds+C\\
&\leq \frac{\mu}{8}\int_{0}^{t}\tilde\sigma\int\frac{\psi_{xx}^{2}}{v}dxds+C,
\end{split}
\end{equation}
and
\begin{equation}
\begin{split}
&\Big| \int_{0}^{t}\tilde\sigma\int R_{1}\psi_{xx}dxds \Big|\\
&\leq \frac{\mu}{8}\int_{0}^{t}\tilde\sigma\int\frac{\psi_{xx}^{2}}{v}dxds+C\int_{0}^{t}\int R_{1}^{2}dxds\\
&\leq \frac{\mu}{8}\int_{0}^{t}\tilde\sigma\int\frac{\psi_{xx}^{2}}{v}dxds+C.\label{62}
\end{split}
\end{equation}
Substituting \eqref{63}-\eqref{62} into \eqref{60}, one obtains 
\begin{gather}
\sup_{0\leq t\leq T}\tilde\sigma\int\psi_{x}^{2}dx+\int_{0}^{T}\tilde\sigma\int\psi_{xx}^{2}dxdt\leq C\max_{x\in\mathbb{R},t\geq0}(\tilde\sigma\theta)+C.\label{12}
\end{gather}
Furthermore, multiplying \eqref{3} by $-\tilde\sigma^{2}\zeta_{xx}$, integrating the resultant over $\mathbb{R}\times (0,t)$, it leads to
\begin{equation}
\begin{split}
&\frac{c_{\nu}}{2}\tilde\sigma^{2}\int\zeta_{x}^{2}dx+\kappa\int_{0}^{t}\tilde\sigma^{2}\int\frac{\zeta_{xx}^{2}}{v}dxds\\
&=c_{\nu}\int_{0}^{t}\tilde\sigma\tilde\sigma^{'}\int\zeta_{x}^{2}dxds
+\int_{0}^{t}\tilde\sigma^{2}\int\big( pu_{x}-p_{+}U_{x} \big)\zeta_{xx}dxds\\
&\quad  -\kappa\int_{0}^{t}\tilde\sigma^{2}\int\zeta_{x}\big( \frac{1}{v} \big)_{x}\zeta_{xx}dxds
-\kappa\int_{0}^{t}\tilde\sigma^{2}\int\big( \frac{\Theta_{x}}{v}-\frac{\Theta_{x}}{V} \big)_{x}\zeta_{xx}dxds\\
&\quad  -\mu\int_{0}^{t}\tilde\sigma^{2}\int\big( \frac{u_{x}^{2}}{v}-\frac{U_{x}^{2}}{V} \big)\zeta_{xx}dxds+\int_{0}^{t}\tilde\sigma^{2}\int R_{2}\zeta_{xx}dxds.\label{64}
\end{split}
\end{equation}
From \eqref{61} and Cauchy's inequality, one has
\begin{equation}
\begin{split}
&c_{\nu}\int_{0}^{t}\tilde\sigma\tilde\sigma_t^{'}\int\zeta_{x}^{2}dxds\leq c_{\nu}\int_{0}^{t}\int\zeta_{x}^{2}dxds\leq C,\label{66}
\end{split}
\end{equation}
and the following estimator is obtained
\begin{equation}
\begin{split}
&\int_{0}^{t}\tilde\sigma^{2}\int\big( pu_{x}-p_{+}U_{x} \big)\zeta_{xx}dxds=\int_{0}^{t}\tilde\sigma^{2}\int\big( \frac{R\theta}{v}\psi_{x}+\frac{R\zeta-p_{+}\phi}{v}U_{x} \big)\zeta_{xx}dxds\\
&\leq \frac{\kappa}{8}\int_{0}^{t}\tilde\sigma^{2}\int\frac{\zeta_{xx}^{2}}{v}dxds
+C\int_{0}^{t}\tilde\sigma^{2}\int\big( (\phi^{2}+\zeta^{2})U_{x}^{2}+\theta^{2}\psi_{x}^{2} \big)dxds\\
&\leq \frac{\kappa}{8}\int_{0}^{t}\tilde\sigma^{2}\int\frac{\zeta_{xx}^{2}}{v}dxds
+C\max_{x\in\mathbb{R},t\geq0}(\tilde\sigma\theta)\int_{0}^{t}\int\theta\psi_{x}^{2}dxds\\
&\quad \quad +C\int_{0}^{t}\max_{x\in\mathbb{R}}\big( \zeta-\frac{\Theta}{2} \big)_{+}^{2}ds+C\int_{0}^{t}\int U_{x}^{2}dxds\\
&\leq \frac{\kappa}{8}\int_{0}^{t}\tilde\sigma^{2}\int\frac{\zeta_{xx}^{2}}{v}dxds+C\max_{x\in\mathbb{R},t\geq0}(\tilde\sigma\theta)+C.
\end{split}
\end{equation}
Recalling \eqref{61} and Cauchy's inequality, one gets
\begin{equation}
\begin{split}
&-\kappa\int_{0}^{t}\tilde\sigma^{2}\int\zeta_{x}\big( \frac{1}{v} \big)_{x}\zeta_{xx}dxds\\
&\leq C\int_{0}^{t}\tilde\sigma^{2}\int\frac{|\zeta_{x}\phi_{x}\zeta_{xx}|+|\zeta_{x}V_{x}\zeta_{xx}|}{v^{2}}dxds\\
&\leq \frac{\kappa}{16}\int_{0}^{t}\tilde\sigma^{2}\int\frac{\zeta_{xx}^{2}}{v}dxds
+C\int_{0}^{t}\tilde\sigma^{2}\int \zeta_{x}^{2}\left(\phi_{x}^{2}+\Theta_{x}^{2} \right)dxds\\
&\leq \frac{\kappa}{16}\int_{0}^{t}\tilde\sigma^{2}\int\frac{\zeta_{xx}^{2}}{v}dxds
+C\int_{0}^{t}\tilde\sigma^{2}\|\zeta_{x}\|_{L^{\infty}(\mathbb{R})}^{2}\int\phi_{x}^{2}dxds
+C\int_{0}^{t}\int\zeta_{x}^{2}dxds\\
&\leq \frac{\kappa}{16}\int_{0}^{t}\tilde\sigma^{2}\int\frac{\zeta_{xx}^{2}}{v}dxds
+C\int_{0}^{t}\tilde\sigma^{2}\|\zeta_{x}\|\|\zeta_{xx}||ds+C\\
&\leq \frac{\kappa}{8}\int_{0}^{t}\tilde\sigma^{2}\int\frac{\zeta_{xx}^{2}}{v}dxds+C,
\end{split}
\end{equation}
by the similarly way, it derives that
\begin{equation}
\begin{split}
&-\kappa\int_{0}^{t}\tilde\sigma^{2}\int\big( \frac{\Theta_{x}}{v}-\frac{\Theta_{x}}{V} \big)_{x}\zeta_{xx}dxds=-\kappa\int_{0}^{t}\tilde\sigma^{2}\int\big( \frac{\Theta_{xx}}{v}-\frac{\Theta_{xx}}{V}-\frac{\Theta_{x}v_{x}}{v^{2}}+\frac{\Theta_{x}V_{x}}{V^{2}} \big)\zeta_{xx}dxds\\
&=-\kappa\int_{0}^{t}\tilde\sigma^{2}\int\big( -\frac{\phi\Theta_{xx}}{vV}-\frac{\phi_{x}\Theta_{x}}{v^{2}}+\frac{\phi(\phi+2V)}{v^{2}V^{2}}\Theta_{x}V_{x} \big)\zeta_{xx}dxds\\
&\leq \frac{\kappa}{8}\int_{0}^{t}\tilde\sigma^{2}\int\frac{\zeta_{xx}^{2}}{v}dxds+C\int_{0}^{t}\tilde\sigma^{2}\int\big( (\Theta_{xx}^{2}+\Theta_{x}^{4})\phi^{2}
+\phi_{x}^{2}\Theta_{x}^{2}+\phi^{4}\Theta_{x}^{4} \big)dxds\\
&\leq \frac{\kappa}{8}\int_{0}^{t}\tilde\sigma^{2}\int\frac{\zeta_{xx}^{2}}{v}dxds
+C\int_{0}^{t}\tilde\sigma^{2}\int\phi_{x}^{2}dxds+C\\
&\leq \frac{\kappa}{8}\int_{0}^{t}\tilde\sigma^{2}\int\frac{\zeta_{xx}^{2}}{v}dxds
+C\int_{0}^{t}\tilde\sigma^{2}\int\zeta^{4}\phi_{x}^{2}dxds+C\int_{0}^{t}\tilde\sigma^{2}\int\theta^{4}\phi_{x}^{2}dxds+C\\
&\leq \frac{\kappa}{8}\int_{0}^{t}\tilde\sigma\int\frac{\zeta_{xx}^{2}}{v}dxds
+C\int_{0}^{t}\|\zeta\|_{L^{\infty}(\mathbb{R})}^{4}\int\phi_{x}^{2}dxds
+C\max_{x\in\mathbb{R},t\geq0}(\tilde\sigma^{2}\theta^{3})\int_{0}^{t}\int\theta\phi_{x}^{2}dxds+C\\
&\leq \frac{\kappa}{8}\int_{0}^{t}\tilde\sigma\int\frac{\zeta_{xx}^{2}}{v}dxds
+C\max_{x\in\mathbb{R},t\geq0}(\tilde\sigma^{2}\theta^{3})+C,\notag
\end{split}
\end{equation}
and
\begin{equation}
\begin{split}
&-\mu\int_{0}^{t}\tilde\sigma^{2}\int\big( \frac{u_{x}^{2}}{v}-\frac{U_{x}^{2}}{V}\zeta_{xx} \big)dxds=-\mu\int_{0}^{t}\tilde\sigma^{2}\int\big( \frac{\psi_{x}^{2}+2\psi_{x}U_{x}}{v}-\frac{\phi U_{x}^{2}}{vV} \big)\zeta_{xx}dxds\\
&\leq \frac{\kappa}{16}\int_{0}^{t}\tilde\sigma^{2}\int\frac{\zeta_{xx}^{2}}{v}dxds
+C\int_{0}^{t}\tilde\sigma^{2}\int\big(\psi_{x}^{4}+\psi_{x}^{2}U_{x}^{2}+\phi^{2}U_{x}^{4}\big)dxds\\
&\leq \frac{\kappa}{16}\int_{0}^{t}\tilde\sigma\int\frac{\zeta_{xx}^{2}}{v}dxds
+C\int_{0}^{t}\tilde\sigma^{2}\int\psi_{x}^{4}dxds+C\int_{0}^{t}\int U_{x}^{4}dxds\\
&\leq \frac{\kappa}{16}\int_{0}^{t}\tilde\sigma\int\frac{\zeta_{xx}^{2}}{v}dxds
+C\int_{0}^{t}\tilde\sigma^{2}||\psi_{x}||_{L^{\infty}(\mathbb{R})}^{2}\int\psi_{x}^{2}dxds+C\\
&\leq \frac{\kappa}{16}\int_{0}^{t}\tilde\sigma^{2}\int\frac{\zeta_{xx}^{2}}{v}dxds
+C\int_{0}^{t}\tilde\sigma\int\big(\psi_{xx}^{2}+\psi_{x}^{2}\big)dxds
+C\\
&\leq \frac{\kappa}{16}\int_{0}^{t}\tilde\sigma^{2}\int\frac{\zeta_{xx}^{2}}{v}dxds
+C\max_{x\in\mathbb{R},t\geq0}(\tilde\sigma\theta)+C.
\end{split}
\end{equation}
From \eqref{43} and Cauchy's inequality, one has
\begin{equation}
\begin{split}
&\int_{0}^{t}\tilde\sigma^{2}\int R_{2}\zeta_{xx}dxds\\
&\leq \frac{\kappa}{16}\int_{0}^{t}\tilde\sigma^{2}\int\frac{\zeta_{xx}^{2}}{v}dxds+C\int_{0}^{t}\int R_{2}^{2}dxds\\
&\leq \frac{\kappa}{16}\int_{0}^{t}\tilde\sigma^{2}\int\frac{\zeta_{xx}^{2}}{v}dxds+C.\label{65}
\end{split}
\end{equation}
Finally, substituting estimates \eqref{66}-\eqref{65} into \eqref{64}, it shows that
\begin{gather}
\tilde\sigma^{2}\int\zeta_{x}^{2}dx+\int_{0}^{T}\tilde\sigma^{2}\int\zeta_{xx}^{2}dxds\leq C\max_{x\in\mathbb{R},t\geq0}(\tilde\sigma^{2}\theta^{3})+C.\label{11}
\end{gather}
By using Sobolev's inequality and \eqref{11}, one obtains
\begin{equation}
    \begin{split}
       \max_{x\in\mathbb{R},t\geq0} (\tilde\sigma\theta^{2})
       &\leq\max_{x\in\mathbb{R},t\geq0}\left(\tilde\sigma \|\zeta\|_{L^{\infty}(\mathbb{R})}^{2}+1\right)\\
        &\leq C\max_{x\in\mathbb{R},t\geq0}\Big(\tilde\sigma\|\zeta\|\|\zeta_{x}\|\Big)+C\\
        &\leq C\max_{x\in\mathbb{R},t\geq0}(\tilde\sigma\theta^{\frac{3}{2}})+C\\
        &\leq \frac{1}{2}\max_{x\in\mathbb{R},t\geq0}(\tilde\sigma\theta^{2})+C,\notag
    \end{split}
\end{equation}
moreover, from a direct calculation, it gives that
\begin{gather}
\tilde\sigma^{\frac{1}{2}}\theta(x,t)\leq C,\quad  (x,t)\in \mathbb{R}\times [0,\infty),\notag
\end{gather}
together with \eqref{12} and \eqref{11}, the proof of Lemma \ref{67} is completed.
\end{proof}

\

With the Lemma 3.1-3.4 at hand, it's time to give the proof of Proposition 3.2. Following from \eqref{68} and \eqref{3},
\begin{gather}
\int_{1}^{\infty}\Big(\| \zeta_{x} \|^{2}+\big| \frac{d}{dt}\| \zeta_{x} \|^{2} \big|\Big)dt\leq C,
\end{gather}
which, along with Sobolev's inequality, it derives that
\begin{gather}
\lim_{t\to\infty}\| \zeta \|_{L^{\infty}(\mathbb{R})}^{2}\leq C\lim_{t\to \infty}\| \zeta \|\| \zeta_{x} \|
\leq C\lim_{t\to \infty}\| \zeta_{x} \|=0.\notag
\end{gather}
Hence, there exists some $T_{0}>0$ such that, for all $(x,t)\in\mathbb{R}\times [T_{0},+\infty)$,  it holds that
\begin{gather}
-\frac{\theta_{-}}{2}<\zeta<\frac{\theta_{+}}{2}.\notag 
\end{gather}
This directly yields
\begin{gather}
\theta=\Theta+\zeta>\theta_{-}-\frac{\theta_{-}}{2}=\frac{\theta_{-}}{2},\quad \forall (x,t)\in\mathbb{R}\times [T_{0},+\infty),\label{xj}
\end{gather}
and
\begin{gather}
    \theta=\Theta+\zeta<\theta_{+}+\frac{\theta_{+}}{2}=\frac{3\theta_{+}}{2}.\label{shang}
\end{gather}
Similarly, one has
\begin{gather}
    \int_{T_{0}}^{\infty}\big( \| \phi_{x} \|_{L^{2}(\mathbb{R})}^{2}
    +\frac{d}{dt} \| \phi_{x} \|_{L^{2}(\mathbb{R})}^{2} \big)dt\leq C,\quad \int_{T_{0}}^{\infty}\big( \| \psi_{x} \|_{L^{2}(\mathbb{R})}^{2}
    +\frac{d}{dt} \| \psi_{x} \|_{L^{2}(\mathbb{R})}^{2} \big)dt\leq C,\notag
\end{gather}
which, directly gives 
\begin{gather}
\lim_{t\to\infty}\big(\| \phi \|_{L^{\infty}(\mathbb{R})}+\| \psi \|_{L^{\infty}(\mathbb{R})}\big)
\leq C\lim_{t\to\infty}\big(\| \phi_{x} \|+\| \psi_{x} \|\big)=0.\notag 
\end{gather}
Due to  $\theta_{-}<\Theta<\theta_{+}$, combining with \eqref{new}, one gets
\begin{gather}
    \Big| \big\{ \frac{\theta}{\theta_{-}}<\frac{1}{2} \big\} \Big|\leq \Big| \big\{ \frac{\theta}{\Theta}<\frac{1}{2} \big\} \Big|\leq C\int \Phi\big(\frac{\theta}{\Theta}\big)dx\leq C,\label{omega1/2}
\end{gather}
where $C$ depending on $C_{0}$. Now, rewriting the equation \eqref{03} as following
\begin{gather}
    c_{\nu}\theta_{t}+\frac{\theta}{v}u_{x}
    =\big( \kappa\frac{\theta_{x}}{v} \big)_{x}+\mu\frac{u_{x}^{2}}{v},\label{1.3}
\end{gather}
for any given $q>2$, multiplying \eqref{1.3} by $\frac{1}{\theta^{2}}\big( \frac{1}{\theta}-\frac{2}{\theta_{-}} \big)_{+}^{q}$, integrating the resultant over $\mathbb{R}$, one obtains
\begin{equation}\label{theta-q}
    \begin{split}
        &\frac{1}{q+1}\frac{d}{dt}\int\Big( \frac{1}{\theta}-\frac{2}{\theta_{-}} \Big)_{+}^{q+1}dx
        +\int\frac{u_{x}^{2}}{v\theta^{2}}\Big( \frac{1}{\theta}-\frac{2}{\theta_{-}} \Big)_{+}^{q}dx\\
        &\leq \int\frac{u_{x}}{v\theta}\Big( \frac{1}{\theta}-\frac{2}{\theta_{-}} \Big)_{+}^{q}dx\\
        &\leq \frac{1}{2}\int\frac{u_{x}^{2}}{v\theta^{2}}\Big( \frac{1}{\theta}-\frac{2}{\theta_{-}} \Big)_{+}^{q}dx
        +\frac{1}{2}\int\frac{1}{v}\Big( \frac{1}{\theta}-\frac{2}{\theta_{-}} \Big)_{+}^{q}dx,
    \end{split}
\end{equation}
from \eqref{theta-q}, combining with \eqref{v}, it leads to 
\begin{equation}
    \begin{split}
        &\Big\| \big( \frac{1}{\theta}-\frac{2}{\theta_{-}} \big)_{+}\Big \|_{L^{q+1}(\mathbb{R})}^{q}
        \frac{d}{dt}\Big\| \big( \frac{1}{\theta}-\frac{2}{\theta_{-}} \big)_{+} \Big \|_{L^{q+1}(\mathbb{R})}\\
        &\leq C\int \big( \frac{1}{\theta}-\frac{2}{\theta_{-}} \big)_{+}^{q}dx\leq C\Big\| \big( \frac{1}{\theta}-\frac{2}{\theta_{-}} \big)_{+}\Big \|_{L^{q+1}(\mathbb{R})}^{q},\notag 
    \end{split}
\end{equation}
where $C$ is independent of $q$, and in particular, it implies that, there exists some positive constants $C=C(\displaystyle\inf_{x\in\mathbb{R}}v_{0},\displaystyle\inf_{x\in\mathbb{R}}\theta_{0},\|\phi_{0}\|_{H^{1}},\|\psi_{0}\|_{L^{2}},T_{0},C_{0})$ which independent of $q$,  such that 
\begin{gather}
    \sup_{0\leq t\leq T_{0}}\Big\| \big( \frac{1}{\theta}-\frac{2}{\theta_{-}} \big)_{+}\Big \|_{L^{q+1}(\mathbb{R})}\leq C.\notag 
\end{gather}
Letting $q\to \infty$, combining with \eqref{omega1/2}, it shows that
\begin{gather}
  \sup_{0\leq t\leq T_{0}}\Big\| \big( \frac{1}{\theta}-\frac{2}{\theta_{-}} \big)_{+}\Big \|_{L^{\infty}(\mathbb{R})}\leq C,\notag
\end{gather}
which together with \eqref{xj}, one gives
\begin{gather}
    \theta\geq \frac{\theta_{-}}{2+C\theta_{-}}>0,\quad \text{for all\ }(x,t)\in\mathbb{R}\times(0,+\infty).\label{26}
\end{gather}
From \eqref{v}, \eqref{26} and Lemma \ref{21}, letting $\alpha=c_{1}/4$, one obtains
\begin{equation}
\begin{split}
    &\int_{0}^{t}\int\big( p_{+}\Phi\big( \frac{V}{v} \big)+\frac{p_{+}}{|\gamma-1|}\Phi
\big( \frac{\Theta}{\theta} \big)
+\frac{|\zeta|}{\theta}|p_{+}-p| \big)|U_{x}|dxds \\
    &\leq \delta C\int_{0}^{t}\frac{1}{1+s}\int (\phi^{2}+\zeta^{2})e^{-\frac{c_{1}x^{2}}{1+s}}dxds\\
&\leq \delta C\big(1+\int_{0}^{t}\int( \phi_{x}^{2}+\psi_{x}^{2}+\zeta_{x}^{2} )dxds\big)\leq \delta C,\label{001}
\end{split}
\end{equation}
and 
\begin{equation}
\begin{split}
    &\int_{0}^{t}\int \big( \frac{\kappa\zeta^{2}\Theta_{x}^{2}}{v\theta^{2}\Theta}
+\frac{\kappa\Theta\phi^{2}\Theta_{x}^{2}}{v\theta^{2}V^{2}}
+\frac{\kappa|\zeta\phi|\Theta_{x}^{2}}{v\theta^{2}V}+\frac{4\mu\zeta^{2}U_{x}^{2}}{v\theta\Theta} +\frac{\mu|\zeta\phi| U_{x}^{2}}{v\theta V}+\frac{\mu\theta\phi^{2}U_{x}^{2}}{v\Theta V^{2}}\big)dxds\\
    &\leq \delta C\int_{0}^{t}\frac{1}{1+s}\int (\phi^{2}+\zeta^{2})e^{-\frac{c_{1}x^{2}}{1+s}}dxds\\
&\leq \delta C\big(1+\int_{0}^{t}\int( \phi_{x}^{2}+\psi_{x}^{2}+\zeta_{x}^{2} )dxds\big)\leq \delta C,\label{002}
\end{split}
\end{equation}
moreover, by the similarly way, one gets
\begin{equation}
\begin{split}
    &\int_{0}^{t}\int \big(|\psi R_{1}|+|\frac{\zeta}{\theta}R_{2}|\big)dxds\\
    &\leq \delta C\int_{0}^{t}\frac{1}{1+s}\int (\psi^{2}+\zeta^{2})e^{-\frac{c_{1}x^{2}}{1+s}}dxds+
    \delta C \int_{0}^{t}\frac{1}{(1+s)^{2}}\int e^{-\frac{c_{1}x^{2}}{1+s}}dxds\\
&\leq \delta C\big(1+\int_{0}^{t}\int( \phi_{x}^{2}+\psi_{x}^{2}+\zeta_{x}^{2} )dxds\big)\leq \delta C,\label{003} 
\end{split}
\end{equation}
 where $C$ depending on $\inf_{x\in \mathbb{R}}v_{0}$, $\inf_{x\in\mathbb{R}}\theta_{0}$, $\|\phi_{0}\|_{H^{1}(\mathbb{R})}$, $\|\psi_{0}\|_{L^{2}(\mathbb{R})}$, $\|\psi_{0}\|_{L^{4}(\mathbb{R})}$, $\|\zeta_{0}\|_{L^{2}(\mathbb{R})}$, $C_{0}$. Following from \eqref{26}-\eqref{003}, it implies that
\begin{eqnarray}\label{bound of G(t)}
        G(t)&\overset{\text{def}}{=} &\int_{0}^{t}\int \big( \frac{\kappa\zeta^{2}\Theta_{x}^{2}}{v\theta^{2}\Theta}
+\frac{\kappa\Theta\phi^{2}\Theta_{x}^{2}}{v\theta^{2}V^{2}}
+\frac{\kappa|\zeta\phi|\Theta_{x}^{2}}{v\theta^{2}V}+\frac{4\mu\zeta^{2}U_{x}^{2}}{v\theta\Theta} +\frac{\mu|\zeta\phi| U_{x}^{2}}{v\theta V}+\frac{\mu\theta\phi^{2}U_{x}^{2}}{v\Theta V^{2}}\big)dxds\notag\\
&& \quad +\int_{0}^{t}\int\big( p_{+}\Phi( \frac{V}{v} )+\frac{p_{+}}{|\gamma-1|}\Phi
( \frac{\Theta}{\theta} )
+\frac{|\zeta|}{\theta}|p_{+}-p| \big)|U_{x}|dxds\\
&& \quad +\int_{0}^{t}\int \big(|\psi R_{1}|+|\frac{\zeta}{\theta}R_{2}|\big)dxds\leq \delta C\leq C_{0}\notag,
 \end{eqnarray}
where in order for the last inequality of \eqref{bound of G(t)} to hold, it need to choose a suitable small $\delta_{0}$ only depending on 
$\inf_{x\in \mathbb{R}}v_{0}$, $\inf_{x\in\mathbb{R}}\theta_{0}$, $\|\phi_{0}\|_{H^{1}(\mathbb{R})}$, $\|\psi_{0}\|_{L^{2}(\mathbb{R})}$, $\|\psi_{0}\|_{L^{4}(\mathbb{R})}$, $\|\zeta_{0}\|_{L^{2}(\mathbb{R})}$,
such that $\delta<\delta_{0}$. So far,  the proof of Proposition 3.2 is completed.
Finally,  the proof of Proposition 3.1 is given below. For $t\in[0,T_{0}]$, multiplying \eqref{2} by $-\psi_{xx}$,  integrating the resultant over $\mathbb{R}\times (0,t)$, one has
\begin{equation}
\begin{split}
&\frac{1}{2}\int\psi_{x}^{2}dx+\mu\int_{0}^{t}\int\frac{\psi_{xx}^{2}}{v}dxds\\
&=\frac{1}{2}\int\psi_{0x}^{2}dx+\int_{0}^{t}\int\big( \frac{R\zeta_{x}}{v}-\frac{R\theta\phi_{x}}{v^{2}}
-\frac{R\zeta-p_{+}\phi}{v^{2}}V_{x} \big)\psi_{xx}dxds\\
&\quad -\mu\int_{0}^{t}\int\psi_{x}\big( \frac{1}{v} \big)_{x}\psi_{xx}dxds
+\mu\int_{0}^{t}\int\big( \frac{U_{x}}{V}-\frac{U_{x}}{v} \big)_{x}\psi_{xx}dxds+\int_{0}^{t}\int R_{1}\psi_{xx}dxds\\
&\leq \frac{\mu}{2}\int_{0}^{t}\int\frac{\psi_{xx}^{2}}{v}dxds+C\max_{(x,t)\in\mathbb{R}\times[0,T_{0}]}\theta+C.\label{130}
\end{split}
\end{equation}
Multiplying \eqref{3} by $\zeta_{xx}$,  integrating the resultant  over $\mathbb{R}\times (0,t)$, one obtains
\begin{equation}
\begin{split}
&\frac{c_{\nu}}{2}\int\zeta_{x}^{2}dx+\kappa\int_{0}^{t}\int\frac{\zeta_{xx}^{2}}{v}dxds\\
&=\frac{c_{\nu}}{2}\int\zeta_{0}^{'2}dx
+\int_{0}^{t}\int\big( pu_{x}-p_{+}U_{x} \big)\zeta_{xx}dxds-\kappa\int_{0}^{t}\int\zeta_{x}\big( \frac{1}{v} \big)_{x}\zeta_{xx}dxds\\
&\qquad  
-\kappa\int_{0}^{t}\int\big( \frac{\Theta_{x}}{v}-\frac{\Theta_{x}}{V} \big)_{x}\zeta_{xx}dxds-\mu\int_{0}^{t}\int\big( \frac{u_{x}^{2}}{v}-\frac{U_{x}^{2}}{V} \big)\zeta_{xx}dxds\\
&\qquad\qquad\qquad+\int_{0}^{t}\int R_{2}\zeta_{xx}dxds.\\
&\leq \frac{\kappa}{2}\int_{0}^{t}\int\frac{\zeta_{xx}^{2}}{v}dxds+C\max_{(x,t)\in\mathbb{R}\times[0,T_{0}]}\theta^{3}+C.\label{560}
\end{split}
\end{equation}
From \eqref{130} and \eqref{560}, it shows that
\begin{equation}
    \begin{split}
        &\sup_{t\in [0,T_{0}]}\int(\psi_{x}^{2}+\zeta_{x}^{2})dx+\int_{0}^{T_{0}}\int(\psi_{xx}^{2}+\zeta_{xx}^{2})dxds\leq C\max_{(x,t)\in\mathbb{R}\times[0,T_{0}]}\theta^{3}+C,\label{820}
    \end{split}
\end{equation}
which together with  Sobolev's inequality and \eqref{61}, one gives
\begin{equation}
\begin{split}
    \max_{(x,t)\in\mathbb{R}\times[0,T_{0}]}\theta^{2}&\leq C\int\zeta\zeta_{x}dx+C\leq C\big(\int\zeta^{2}dx\big)^{\frac{1}{2}}\big(\int\zeta_{x}^{2}dx\big)^{\frac{1}{2}}+C\\
    &\leq C\max_{(x,t)\in\mathbb{R}\times[0,T_{0}]}\theta^{\frac{3}{2}}+C\leq \frac{1}{2}\max_{(x,t)\in\mathbb{R}\times[0,T_{0}]}\theta^{2}+C,\notag
\end{split}
\end{equation}
which in particular implies that
\begin{gather}
   \max_{(x,t)\in\mathbb{R}\times[0,T_{0}]}\theta\leq C(T_{0}),\notag
\end{gather}
and moreover, from \eqref{shang}, it holds that
\begin{gather}
     \sup_{(x,t)\in\mathbb{R}\times[0,+\infty)}\theta\leq C,\notag
\end{gather}
 together with lemma 2.3 and  \eqref{820}, it gives 
\begin{gather}
    \sup_{t\in(0,+\infty)}\int(\psi_{x}^{2}+\zeta_{x}^{2})dx+\int_{0}^{+\infty}\int(\psi_{xx}^{2}+\zeta_{xx}^{2})dxds\leq C,\notag
\end{gather}
and so far,   the proof of proposition 3.1 is completed.

\section{Proof of Theorem 1.2}
\indent\qquad
Similar to the proof procedure in section 3, noticing that $(V_{\pm}^{r},U_{\pm}^{r},\Theta_{\pm}^{r})$ satisfies Euler system \eqref{14} and $(V^{cd},U^{cd},\Theta^{cd})$ satisfies  \eqref{931},   rewriting the Cauchy problem \eqref{01}-\eqref{23} as
\begin{gather}
 \displaystyle      \phi_{t}-\psi_{x}=0 ,\label{3301}\\
 \displaystyle      \psi_{t}+\left( p-P \right)_{x}=\mu\big( \frac{u_{x}}{v}-\frac{U_{x}}{V} \big)_{x}+F ,\label{3302} \\
\displaystyle        c_{\nu}\zeta_{t}+pu_{x}-PU_{x}=\kappa\big( \frac{\theta_{x}}{v}-\frac{\Theta_{x}}{V} \big)_{x}+\mu\big( \frac{u_{x}^{2}}{v}-\frac{U_{x}^{2}}{V} \big)+G,\label{3303}\\
\displaystyle      (\phi,\psi,\zeta)(x,0)=(\phi_{0},\psi_{0},\zeta_{0})(x)\xrightarrow{x\rightarrow\pm\infty}(0,0,0),\quad x\in \mathbb{R},\label{3305}
\end{gather}
where 
\begin{equation}
\begin{split}
&P=\frac{R\Theta}{V},\quad P_{\pm}=\frac{R\Theta_{\pm}^{r}}{V_{\pm}^{r}},\quad F=(P_{-}+P_{+}-P)_{x}+\big( \mu\frac{U_{x}}{V} \big)_{x}-U_{t}^{cd},\notag
\end{split}
\end{equation}
and 
\begin{equation}
\begin{split}
G=(p^{m}&-P)U_{x}^{cd}+(P_{-}-P)(U_{-}^{r})_{x}+(P_{+}-P)(U_{+}^{r})_{x}+\mu\frac{U_{x}^{2}}{V}+\kappa\big( \frac{\Theta_{x}}{V}-\frac{\Theta_{x}^{cd}}{V^{cd}} \big)_{x}.\notag
\end{split}
\end{equation}

\begin{proposition}
There exists a small constant $\delta_{0}$ only depending on $ \inf_{x\in \mathbb{R}}v_{0}$, $\inf_{x\in\mathbb{R}}\theta_{0}$, $\|\phi_{0}\|_{H^{1}(\mathbb{R})}$, $\|\psi_{0}\|_{L^{2}(\mathbb{R})}$, $\|\psi_{0}\|_{L^{4}(\mathbb{R})}$, $\|\zeta_{0}\|_{L^{2}(\mathbb{R})}$,
such that  $D(t)\leq C_{0}$ provided $D(t)\leq 2C_{0} $ and $\delta<\delta_{0}$, where
\begin{eqnarray}
 D(t)&=&\int_{0}^{t}\int\big(\frac{\kappa\zeta^{2}}{v\theta^{2}\Theta}+\frac{\kappa\Theta\phi^{2}}{v\theta^{2}V^{2}}+\frac{\kappa|\phi\zeta|}{v\theta^{2}V}\big)\Theta_{x}^{2}dxds+\int_{0}^{t}\int\big( \frac{\mu\theta\phi^{2}}{v\theta V^{2}}+\frac{4\mu\zeta^{2}}{v\theta\Theta}+\frac{\mu|\phi\zeta|}{v\theta V} \big)U_{x}^{2}dxds\notag\\
&& +\int_{0}^{t}\int \big( |F\psi|+|G\frac{\zeta}{\theta}| \big)dxds+\int_{0}^{t}\int|Q_{2}|dxds,\label{901}
\end{eqnarray}
and
\begin{gather}
    C_{0}=\int\big(\frac{\psi_{0}^{2}}{2}+R\Theta_{0}\Phi( \frac{v_{0}}{V_{0}})+c_{\nu}\Theta_{0}\Phi( \frac{\theta_{0}}{\Theta_{0}}  )\big)dx.\label{7312}
\end{gather}
\end{proposition}

\vskip 0.3cmTo prove proposition 4.1, similar as Section 3,  the following series lemmas are needed.
\begin{lemma}
Under the Proposition 4.1  that $D(t)\leq 2C_{0}$, it holds
\begin{equation}
\begin{split}
&\int\Big( \frac{\psi^{2}}{2}+R\Theta\Phi\big( \frac{v}{V} \big)+\frac{R}{\gamma-1}\Theta\Phi\big( \frac{\theta}{\Theta} \big) \Big)dx+\frac{1}{2}\int_{0}^{t}\int\big(\frac{\mu\Theta}{v\theta}\psi_{x}^{2}+\frac{\kappa\Theta}{v\theta^{2}}\zeta_{x}^{2} \big)dxds\\
&\quad +\int_{0}^{t}\int Q_{1}\big( (U_{-}^{r})_{x}+(U_{+}^{r})_{x} \big)dxds\\
&\leq \int_{0}^{t}\int\big(\frac{\kappa\zeta^{2}}{v\theta^{2}\Theta}+\frac{\kappa\Theta\phi^{2}}{v\theta^{2}V^{2}}+\frac{\kappa|\phi\zeta|}{v\theta^{2}V}\big)\Theta_{x}^{2}dxds+\int_{0}^{t}\int\big( \frac{\mu\theta\phi^{2}}{v\theta V^{2}}+\frac{4\mu\zeta^{2}}{v\theta\Theta}+\frac{\mu|\phi\zeta|}{v\theta V} \big)U_{x}^{2}dxds\\
&\quad +\int_{0}^{t}\int \big( |F\psi|+|G\frac{\zeta}{\theta}| \big)dxds+\int\Big( \frac{\psi_{0}^{2}}{2}+R\Theta_{0}\Phi\big( \frac{v_{0}}{V_{0}} \big)+\frac{R}{\gamma-1}\Theta_{0}\Phi\big( \frac{\theta_{0}}{\Theta_{0}} \big) \Big)dx\\
&\quad \quad+\int_{0}^{t}\int|Q_{2}|dxds\leq 3C_{0},\label{4.1}
\end{split}
\end{equation}
where $Q_{1},Q_{2}$ is defined by \eqref{4113}, \eqref{905} and $C_{0}$ is the same as \eqref{7312}.
\end{lemma}
\begin{proof}
Multiplying \eqref{3302} by $\psi$, it leads that
\begin{equation}
\begin{split}
F\psi=&\big(\frac{\psi^{2}}{2}\big)_{t}+\Big[ (p-P)\psi-\mu\big( \frac{u_{x}}{v}-\frac{U_{x}}{V} \big)\psi \Big]_{x}-\frac{R\zeta}{v}\psi_{x}\\
&-R\Theta\big( \frac{1}{v}-\frac{1}{V} \big)\phi_{t}+\mu\frac{\psi_{x}^{2}}{v}+\mu\big( \frac{1}{v}-\frac{1}{V} \big)U_{x}\psi_{x},
\end{split}
\end{equation}
multiplying \eqref{3303} by $\zeta\theta^{-1}$, one gets
\begin{equation}
\begin{split}
G\frac{\zeta}{\theta}=&\frac{R}{\gamma-1}\frac{\zeta\zeta_{t}}{\theta}-\Big[ \kappa\big( \frac{\theta_{x}}{v}-\frac{\Theta_{x}}{V} \big)\frac{\zeta}{\theta} \Big]_{x}+\frac{R\zeta}{v}\psi_{x}\\
&+(p-P)U_{x}\frac{\zeta}{\theta}+\kappa\frac{\Theta\zeta_{x}^{2}}{v\theta^{2}}-\kappa\frac{\zeta\zeta_{x}\Theta_{x}}{v\theta^{2}}-\kappa\frac{\phi\Theta\Theta_{x}\zeta_{x}}{v\theta^{2}V}\\
&+\kappa\frac{\phi\zeta\Theta_{x}^{2}}{v\theta^{2}V}-\mu\frac{\zeta\psi_{x}^{2}}{v\theta}-2\mu\frac{\zeta U_{x}\psi_{x} }{v\theta}+\mu\frac{\phi\zeta U_{x}^{2}}{v\theta V}.\label{903}
\end{split}
\end{equation}
Noting that
\begin{gather}
-R\Theta\big( \frac{1}{v}-\frac{1}{V} \big)\phi_{t}=\Big[ R\Theta\Phi\big(\frac{v}{V}\big) \Big]_{t}-R\Theta_{t}\Phi\big(\frac{v}{V}\big)+\frac{P\phi^{2}}{vV}V_{t},
\end{gather}

\begin{gather}
    \frac{\zeta\zeta_{t}}{\theta}=\Big[\Theta\Phi\big( \frac{\theta}{\Theta} \big)\Big]_{t}+\Theta_{t}\Phi\big(\frac{\Theta}{\theta}\big),\label{904}
\end{gather}

\begin{equation}
\begin{split}
-R\Theta_{t}&=(\gamma-1)P_{-}(U_{-}^{r})_{x}+(\gamma-1)P_{+}(U_{+}^{r})_{x}-p^{m}U_{x}^{cd}\\
&=(\gamma-1)P(U_{-}^{r})_{x}+(\gamma-1)P(U_{+}^{r})_{x}+(\gamma-1)(P_{-}-P)(U_{-}^{r})_{x}\\
&\quad +(\gamma-1)(P_{+}-P)(U_{+}^{r})_{x}-p^{m}U_{x}^{cd},\notag
\end{split}
\end{equation}
and 
\begin{equation}
\begin{split}
&Q_{1}\big((U_{-}^{r})_{x}+(U_{+}^{r})_{x}\big)+Q_{2}\\
&=-R\Theta_{t}\Phi\big(\frac{v}{V}\big)+\frac{P\phi^{2}}{vV}V_{t}+\frac{R}{\gamma-1}\Theta_{t}\Phi\big(\frac{\Theta}{\theta}\big)+(p-P)U_{x}\frac{\zeta}{\theta},\notag
\end{split}
\end{equation}
where
\begin{equation}
\begin{split}
Q_{1}&=(\gamma-1)P\Phi\big( \frac{v}{V} \big)+\frac{P\phi^{2}}{vV}-P\Phi\big( \frac{\Theta}{\theta} \big)+\frac{\zeta}{\theta}(p-P)=P\Big( \Phi\big( \frac{\theta V}{v\Theta} \big)+\gamma\Phi\big( \frac{v}{V} \big) \Big),\label{4113}
\end{split}
\end{equation}
and 
\begin{equation}
\begin{split}
Q_{2}=&U_{x}^{cd}\Big(
\frac{P\phi^{2}}{vV}-p^{m}\Phi\big( \frac{v}{V} \big)+\frac{p^{m}}{\gamma-1}\Phi\big( \frac{\Theta}{\theta} \big)+\frac{\zeta}{\theta}(p-P)
\Big)\\
&+(\gamma-1)(P_{-}-P)(U_{-}^{r})_{x}\Big(
\Phi\big( \frac{v}{V}\big) -\frac{1}{\gamma-1}\Phi\big( \frac{\Theta}{\theta}\big)
\Big)\\
&+(\gamma-1)(P_{+}-P)(U_{+}^{r})_{x}\Big(
\Phi\big( \frac{v}{V}\big) -\frac{1}{\gamma-1}\Phi\big( \frac{\Theta}{\theta}\big)
\Big).\label{905}
\end{split}
\end{equation}
Combining \eqref{903} and \eqref{904},  following from \eqref{904}-\eqref{905}, one has
\begin{gather}
\Big( \frac{\psi^{2}}{2}+R\Theta\Phi\big( \frac{v}{V} \big)+\frac{R}{\gamma-1}\Theta\Phi\big( \frac{\theta}{\Theta} \big) \Big)_{t}+\frac{\mu\Theta}{v\theta}\psi_{x}^{2}+\frac{\kappa\Theta}{v\theta^{2}}\zeta_{x}^{2}\label{96854}\\
+H_{x}+Q_{1}\big( (U_{-}^{r})_{x}+(U_{+}^{r})_{x} \big)+\bar{Q}=F\psi+G\frac{\zeta}{\theta},\notag
\end{gather}
where
\begin{equation}
\begin{split}
\bar{Q}=&Q_{2}-\kappa\frac{\zeta\zeta_{x}\Theta_{x}}{v\theta^{2}}
-\kappa\frac{\phi\Theta\Theta_{x}\zeta_{x}}{v\theta^{2}V}+\kappa\frac{\phi\zeta\Theta_{x}^{2}}{v\theta^{2}V}-\mu\frac{\phi\psi_{x}U_{x}}{vV}-2\mu\frac{\zeta\psi_{x}U_{x}}{v\theta}+\mu\frac{\phi\zeta U_{x}^{2}}{v\theta V}.
\end{split}
\end{equation}
Recalling  Lemma 2.5, one gives
\begin{equation}
\begin{split}
&|(P_{-}-P)(U_{-}^{r})_{x}|\\
&\leq C\Big(
|\Theta^{cd}-\theta_{-}^{m}|+|\Theta_{+}^{r}-\theta_{+}^{m}|+|V^{cd}-v_{-}^{m}|+|V_{+}^{r}-v_{+}^{m}|
\Big)\big|(U_{-}^{r})_{x}\big|\\
&\leq C \Big(
|\Theta^{cd}-\theta_{-}^{m}|+|\Theta_{+}^{r}-\theta_{+}^{m}|+|V^{cd}-v_{-}^{m}|+|V_{+}^{r}-v_{+}^{m}|\Big)\Big|_{\Omega_{-}}+C\big|(U_{-}^{r})_{x}\big|\Big|_{\Omega_{c}\cap\Omega_{+}}\\
&\leq C\delta e^{-c_{0}(|x|+t)}.\label{4.18}
\end{split}
\end{equation}
Thus, one obtains 
\begin{equation}
\begin{split}
|Q_{2}|&\leq C\Big(
\frac{\phi^{2}}{v}+\Phi\big(\frac{v}{V}\big)+\Phi\big(\frac{\Theta}{\theta}\big)+\frac{\phi^{2}+\zeta^{2}}{v\theta}
\Big)|U_{x}^{cd}| +C\delta e^{-c_{0}(|x|+t)}\Big(
\Phi\big(\frac{v}{V}\big)+\Phi\big(\frac{\Theta}{\theta}\big)
\Big),
\end{split}
\end{equation}
and 
\begin{equation}
\begin{split}
|Q|&\leq \frac{1}{2}\big(
\frac{\mu\Theta}{v\theta}\psi_{x}^{2}+\frac{\kappa\Theta}{v\theta^{2}}\zeta_{x}^{2}
\big)+\big(\frac{\kappa\zeta^{2}}{v\theta^{2}\Theta}+\frac{\kappa\Theta\phi^{2}}{v\theta^{2}V^{2}}+\frac{\kappa|\phi\zeta|}{v\theta^{2}V}\big)\Theta_{x}^{2}\\
&\quad +\big( \frac{\mu\theta\phi^{2}}{v\Theta V^{2}}+\frac{4\mu\zeta^{2}}{v\theta\Theta}+\frac{\mu|\phi\zeta|}{v\theta V} \big)U_{x}^{2}+|Q_{2}|.\label{4.20}
\end{split}
\end{equation}
Integrating \eqref{96854} over $\mathbb{R}\times(0,t)$, by using \eqref{4.18}-\eqref{4.20}, one gets
\begin{eqnarray}
&&\int\Big( \frac{\psi^{2}}{2}+R\Theta\Phi\big( \frac{v}{V} \big)+\frac{R}{\gamma-1}\Theta\Phi\big( \frac{\theta}{\Theta} \big) \Big)dx+\frac{1}{2}\int_{0}^{t}\int\big( \frac{\mu\Theta}{v\theta}\psi_{x}^{2}+\frac{\kappa\Theta}{v\theta^{2}}\zeta_{x}^{2} \big)dxds\notag\\
&&\quad +\int_{0}^{t}\int Q_{1}\big( (U_{-}^{r})_{x}+(U_{+}^{r})_{x} \big)dxds\notag\\
&&\leq \int_{0}^{t}\int\big(\frac{\kappa\zeta^{2}}{v\theta^{2}\Theta}+\frac{\kappa\Theta\phi^{2}}{v\theta^{2}V^{2}}+\frac{\kappa|\phi\zeta|}{v\theta^{2}V}\big)\Theta_{x}^{2}dxds+\int_{0}^{t}\int\big( \frac{\mu\theta\phi^{2}}{v\theta V^{2}}+\frac{4\mu\zeta^{2}}{v\theta\Theta}+\frac{\mu|\phi\zeta|}{v\theta V} \big)U_{x}^{2}dxds\notag\\
&&\quad +\int_{0}^{t}\int \big( |F\psi|+|G\frac{\zeta}{\theta}| \big)dxds+\int\Big( \frac{\psi_{0}^{2}}{2}+R\Theta_{0}\Phi\big( \frac{v_{0}}{V_{0}} \big)+\frac{R}{\gamma-1}\Theta_{0}\Phi\big( \frac{\theta_{0}}{\Theta_{0}} \big) \Big )dx\notag\\
&&\quad +\int_{0}^{t}\int|Q_{2}|dxds\leq 3C_{0},\notag
\end{eqnarray}
which \eqref{4.1} is derived, and  the proof of lemma 4.1 is completed.
\end{proof}
 
 \vskip 0.3cm
Similarity as the proof in  Section 3, the following estimates hold:
\begin{gather}
    \int\big(\phi_{x}^{2}+\zeta^{2}+\psi^{4}\big)dx
+\int_{0}^{T}\int\Big(\big(\psi_{x}^{2}+\phi_{x}^{2}\big)\theta
+\psi^{2}\psi_{x}^{2}+\zeta_{x}^{2}\Big)dxdt\leq C,\label{as}
\end{gather}
\begin{gather}
    \tilde\sigma\int\psi_{x}^{2}dx+\tilde\sigma^{2}\int\zeta_{x}^{2}dx+\int_{0}^{T}\tilde\sigma\int\psi_{xx}^{2}dxds
+\int_{0}^{T}\tilde\sigma^{2}\int\zeta_{xx}^{2}dxds\leq C,
\end{gather}
and
\begin{gather}
    C^{-1}\leq v(x,t)\leq C,\quad C^{-1}\leq \theta(x,t),\quad (x,t)\in\mathbb{R}\times (0,+\infty),\label{ass}
\end{gather}
where $C$ depending only on $\inf_{x\in \mathbb{R}}v_{0}$, $\inf_{x\in\mathbb{R}}\theta_{0}$, $\|\phi_{0}\|_{H^{1}(\mathbb{R})}$, $\|\psi_{0}\|_{L^{2}(\mathbb{R})}$, $\|\psi_{0}\|_{L^{4}(\mathbb{R})}$, $\|\zeta_{0}\|_{L^{2}(\mathbb{R})}$, $C_{0}$.

\vskip 0.3cm
{\bf Proof of Proposition 4.1.}
First, recalling Lemma 2.4, Lemma 2.5, by using \eqref{as}-\eqref{ass}, it holds that
\begin{equation}
\begin{split}
& \int_{0}^{t}\int\big(\frac{\kappa\zeta^{2}}{v\theta^{2}\Theta}+\frac{\kappa\Theta\phi^{2}}{v\theta^{2}V^{2}}+\frac{\kappa|\phi\zeta|}{v\theta^{2}V}\big)\Theta_{x}^{2}dxds+\int_{0}^{t}\int\big( \frac{\mu\theta\phi^{2}}{v\Theta V^{2}}+\frac{4\mu\zeta^{2}}{v\theta\Theta}+\frac{\mu|\phi\zeta|}{v\theta V} \big)U_{x}^{2}dxds\\
&\leq C\int_{0}^{t}\int(\phi^{2}+\zeta^{2})(\Theta_{x}^{2}+U_{x}^{2})dxds\\
&\leq  C\int_{0}^{t}\int(\phi^{2}+\zeta^{2})((\Theta_{x}^{cd})^{2}+
(U_{-}^{r})_{x}^{2}+(U_{+}^{r})_{x}^{2})dxds\\
&\leq C\delta\int_{0}^{t}\int\frac{\phi^{2}+\zeta^{2}}{1+s}e^{-c_{1}x^{2}/(1+s)}dxds+C\delta\int_{0}^{t}\int(\phi^{2}+\zeta^{2})e^{-c_{0}(|x|+s)}dxds\\
&\leq C\delta\int_{0}^{t}\int\frac{\phi^{2}+\zeta^{2}}{1+s}e^{-c_{1}x^{2}/(1+s)}dxds+C\delta.\label{921}
\end{split}
\end{equation}
Next, by direct calculation, one has
\begin{gather}
    \|(F,G)\|_{L^{1}(\mathbb{R})}\leq C\delta^{1/8}(1+t)^{-7/8},
\end{gather}
which along with \eqref{as}-\eqref{ass} implies that
\begin{equation}
\begin{split}
&\int_{0}^{t}\int \big( |F\psi|+\big|G\frac{\zeta}{\theta}\big| \Big)dxds\\
&\leq C\int_{0}^{t}\big(\|\psi\|_{L^\infty(\mathbb{R})}+\|\zeta\|_{L^\infty(\mathbb{R})}\big)\int\big(|F|+|G|\big)dxds\\
&\leq C\delta\int_{0}^{t}\frac{1}{(1+s)^{7/8}}\left(\|\zeta_{x}\|^{1/2}+\|\psi_{x}\|^{1/2}\right)ds\\
&\leq C\delta\int_{0}^{t}\frac{1}{(1+s)^{7/6}}ds+C\delta\int_{0}^{t}\int(\psi_{x}^{2}+\zeta_{x}^{2})dxds\leq C\delta.
\end{split}
\end{equation}
Moreover, duo to \eqref{as}-\eqref{ass}, one gets
\begin{equation}
\begin{split}
&\int_{0}^{t}\int|Q_{2}|dxds\\
&\leq C\int_{0}^{t}\int(\phi^{2}+\zeta^{2})|U_{x}^{cd}|dxds+ C\delta\int_{0}^{t}\int(\phi^{2}+\zeta^{2})e^{-c_{0}(|x|+s)}dxds\\
&\leq C\delta\int_{0}^{t}\int\frac{\phi^{2}+\zeta^{2}}{1+s}e^{-c_{1}x^{2}/(1+s)}dxds+C\delta.
\end{split}
\end{equation}
Finally, noting that
\begin{equation}
\begin{split}
&\int_{0}^{t}\int\frac{\phi^{2}+\zeta^{2} }{1+s}e^{-\frac {c_{1}x^{2}}{1+s}}dxds\\
&\leq C\int_{0}^{t}\int(\phi_{x}^{2}+\psi_{x}^{2}+\zeta_{x}^{2})dxds+C\int_{0}^{t}\int(\phi^{2}+\zeta^{2})((U_{-}^{r})_{x}+(U_{+}^{r})_{x})dxds+C\\
&\leq C_0,\label{922}
\end{split}
\end{equation}
combining with \eqref{921}-\eqref{922}, one obtains
\begin{equation}
\begin{split}
 D(t)=&\int_{0}^{t}\int\big(\frac{\kappa\zeta^{2}}{v\theta^{2}\Theta}+\frac{\kappa\Theta\phi^{2}}{v\theta^{2}V^{2}}+\frac{\kappa|\phi\zeta|}{v\theta^{2}V}\big)\Theta_{x}^{2}dxds+\int_{0}^{t}\int\big( \frac{\mu\theta\phi^{2}}{v\theta V^{2}}+\frac{4\mu\zeta^{2}}{v\theta\Theta}+\frac{\mu|\phi\zeta|}{v\theta V} \big)U_{x}^{2}dxds\\
& \quad +\int_{0}^{t}\int \big( |F\psi|+|G\frac{\zeta}{\theta}| \big)dxds+\int_{0}^{t}\int|Q_{2}|dxds\leq \delta C\leq C_{0}.\notag
\end{split}
\end{equation}
where the last inequality is obtained by choosing a sufficiently small $\delta_{0}$ which only depending on 
$\inf_{x\in \mathbb{R}}v_{0}$, $\inf_{x\in\mathbb{R}}\theta_{0}$, $\|\phi_{0}\|_{H^{1}(\mathbb{R})}$,  $\|\psi_{0}\|_{L^{2}(\mathbb{R})}$, $\|\psi_{0}\|_{L^{4}(\mathbb{R})}$, $\|\zeta_{0}\|_{L^{2}(\mathbb{R})}$,
such that  $\delta<\delta_{0}$. As so far, the proof of Proposition 4.1 is completed. It is easy to check that the other estimates for a single viscous contact wave still hold for the case in which the composite waves are the combination of viscous contact wave with rarefaction waves. Thus we finish the proof of Proposition 3.1, and so the proof of Theorem 1.2 is completed.

\vskip 0.3cm {\bf Conflict of Interest}\quad The authors declare that they have no conflict of interest.


\begin{thebibliography}{99}
\bibitem{5a}Duan, R., Guo, A., Zhu, C. J.: Global strong solution to compressible Navier-Stokes
	equations with density dependent viscosity and temperature dependent heat conductivity. J. Differential Equations $\mathbf{262}$ (2017), 4314-4335.

\bibitem{5bb}Huang.  F. M. , Wang. T.: Stability of superposition of viscous contact wave and rarefaction waves for compressible Navier-Stokes system. Indiana University Mathematics Journal, Vol. 65, No. 6(2016).

\bibitem{HLM2010}Huang.  F. M. , Li. J., and  Matsumura, A.: Asymptotic stability of combination of viscous contact wave with rarefaction waves for one-dimensional compressible Navier-Stokes system, Arch. Ration. Mech. Anal. 197(2010), no. 1, 89-116.

\bibitem{5c}Huang.  F. M.  and  Matsumura, A.: Stability of a composite wave of two viscous shock wave for the full compressible Navier-Stokes equation, Comm. Math. Phys. 289(2009), no. 3, 841-864.

\bibitem{HMS2004} Huang.   F. M. ,  Matsumura, A. and Shi. X.D.: On the stability of contact discontinuity for compressible Navier-Stokes equations with free boundary, Osaka J. Math. 41(2004), no. 1, 193-210.
    
\bibitem{5e}Huang.   F. M. , Matsumura, A. and Xin. Z.P.: Stability of contact discontinuities for the 1-D compressible Navier-Stokes equations, Arch. Ration. Mech. Anal. 179(2006), no. 1, 55-77.
    
\bibitem{5f}Huang.   F. M., Xin. Z.P., and Yang. T.: Contact discontinuity with general perturbations for gas motions, Adv. Math. 219(2008), no. 4, 1246-1297.
    
\bibitem{5h}Huang. F. M.,  Zhao. H.J.: On the global stability of contact discontinuity for compressible Navier-Stokes equations, Rend. Sem. Mat. Univ. Padova 109(2003), 283-305.

\bibitem{5}Hong. H. H.: Global stability of viscous contact wave for 1-D compressible Navier-Stokes equations, J. Differential Equations 252(2012), no. 5, 3482-3505.

\bibitem{6}Huang B., Shi X. D.:  Nonlinearly exponential stability of compressible Navier-Stokes system with degenerate heat-conductivity.  J. Differential Equations $\mathbf{268}$ (2020),    2464-2490.

\bibitem{7}Huang, B., Shi, X. D.,    Sun, Y.: Global strong solution to compressible Navier-Stokes system with degenerate heat conductivity
	and density-depending viscosity. Commun. Math. Sci. {\bf 18} (2020), 973-985.
	
\bibitem{8}	J. Goodman.: Nonlinear asymptotic stability of viscous shock profiles for conservation laws, Arch. Rational Mech. Anal. 95(1986), no.4,325-344.
	
\bibitem{9}Jiang, S.: Large-time behavior of solutions to the equations of a one-dimensional viscous polytropic ideal gas in unbounded domains, Comm. Math. Phys. 200(1999), no. 1, 181-193.
	
\bibitem{10}Jiang,  S.: Remarks on the asymptotic behaviour of solutions to the compressible Navier-Stokes equations in the half-line, Proc. Roy. Soc. Edinburgh Sect. A 132(2002), no. 3, 627-638.
	
\bibitem{smoller}J.Smoller.: Shock wave and Reaction-diffusion equations, 2nd ed., Springer Verlag, New York, 1994.
	
\bibitem{12}Kawashima,  S., Matsumura, A. : Asymptotic stability of traveling wave solutions of systems for one-dimensional gas motion, Comm. Math. Phys. 101 (1985), no. 1. 97-127.

\bibitem{13}Kawashima, S., Nishida, T.: Global solutions to the initial value problem for the equations of one dimensional motion of viscous polytropic gases. J. Math. Kyoto Univ. 21 (1981),
	825-837.
	
\bibitem{15a}Kazhikhov, A. V., Shelukhin, V. V.: Unique global solution with respect to time of initial boundary value problems for one-dimensional equations of a viscous gas. J. Appl.
	Math. Mech. 41 (1977), 273-282.
	
\bibitem{15}L. Hsiao and T.P. Liu, Nonlinear diffusive phenomena of nonlinear hyperbolic systems, Chinese Ann. Math. Ser. B 14(1933), no. 4, 465-480. A Chinese summary appears in Chinese Ann. Math. Ser A 14(1993), no. 6, p. 740. Mr1254543

\bibitem{16} Li, J.,   Liang, Z. L.: Some uniform estimates and large time behavior of solutions to one dimensional compressible navier stokes system
    in unbounded domain with large data. Arch. Rat. Mech. Anal. 220(2016), 1195-1208.
    
\bibitem{opq}Matsumura, A., Nishihara, K.: On the stability of travelling wave solutions of a one-dimensional model system for compressible viscous gas, Japan J. Appl. Math. 2(1985), no. 1, 17-25.
    
\bibitem{mats}Matsumura, A., Nishihara, K.: Asymptotics toward the rarefaction waves of the solutions of a one-dimensional model system for compressible viscous gas. Jpan J. Appl. Math. 3, 1-13(1986).
    
\bibitem{hny}Matsumura, A., Nishihara, K.: Global stability of the rarefaction waves of  a one-dimensional model system for compressible viscous gas. Comm. Math. Phys. 144 (1992), no. 2, 325-335.
    
	
\bibitem{Tyang}Nishihara, K., Yang, T., Zhao, H.J.: Nonlinear stability of strong rarefaction waves for compressible Navier-Stokes equations, Siam J. Math Anal. 35 (2004), no. 6, 1561-1597.
	
\bibitem{23}Pan, R. H., Zhang, W. Z.: Compressible Navier-Stokes equations with temperature dependent heat conductivities, Commun. Math. Sci. 13 (2015), 401-425.

\end{thebibliography}
\end{document}